\documentclass[12pt,a4paper]{article}
\usepackage[utf8]{inputenc}
\usepackage{amsmath,amsthm}
\usepackage{mathrsfs}
\usepackage[dvipsnames]{xcolor}
\usepackage{amsfonts}
\usepackage{amssymb}
\usepackage{tikz}
\usepackage{tikz-cd}
\usepackage{caption}
\usepackage[nottoc]{tocbibind}
\usepackage[toc,page]{appendix}
\usepackage{listings}
\usepackage{pdfpages}
\usepackage{mathtools}
\usepackage[all,cmtip]{xy}
\usepackage{lipsum}
\usepackage{breqn}
\usepackage{subcaption}
\usepackage{hyperref}
\usepackage{tkz-euclide}
\usepackage{pgfplots}
\usepackage[sort,nocompress]{cite}
\usepackage{comment}

\usetikzlibrary{angles,quotes}
\usetikzlibrary{decorations.markings}
\usetikzlibrary{hobby}

\usepackage{authblk}

\newtheorem{corollary}{Corollary}
\newtheorem{corollary2}{Corollary}
\newtheorem{claim}{Claim}
\newtheorem{theorem}{Theorem}
\newtheorem{definition}{Definition}
\newtheorem{lemma}{Lemma}

\newtheorem{proposition}{Proposition}

\theoremstyle{remark}
\newtheorem*{remark}{Remark}

\numberwithin{lemma}{section}
\numberwithin{claim}{section}
\numberwithin{corollary2}{section}
\numberwithin{definition}{section}
\numberwithin{proposition}{section}
\numberwithin{equation}{section}

\def\N{\mathbb{N}}
\def\Z{\mathbb{Z}}
\def\C{\mathbb{C}}
\def\R{\mathbb{R}}
\def\Q{\mathbb{Q}}

\def\<={\leq}
\def\>={\geq}
\def\inv{^{-1}}

\def\im{\mathrm{Im}\,}

\pgfplotsset{compat=1.18}

\newcommand{\set}[1]{\left\lbrace#1\right\rbrace}

\newcommand{\norm}[1]{\lVert #1 \rVert}
\newcommand{\normC}[2]{\norm{#1}_{C^{#2}}}
\newcommand{\normL}[2]
{\norm{#1}_{L^{#2}}}

\title{Rigidity of the Suris' 
potential in the Frenkel-Kontorova Model}

\author[1]{Corentin Fierobe\thanks{University of Rome Tor Vergata.
    \textit{E-mail}: \href{cpef@gmx.de}{\texttt{cpef@gmx.de}}}\hspace{2mm}  and Daniel Tsodikovich\thanks{  Institute of Science and Technology Austria
   Am Campus 1, Klosterneuburg 3400, Austria.
    \textit{E-mail}: \href{daniel.tsodikovich@ist.ac.at}{\texttt{daniel.tsodikovich@ist.ac.at}}}}

\date{}
\begin{document}

\maketitle
\begin{abstract}
The goal of this paper is to establish a local rigidity result for the integrability of standard-like maps.
The main focus of the paper is the remarkable integrable potential discovered by Suris in the 80's. 
We show that locally, the integrability of this potential is rigid.
The proof relies on a similar strategy that was used for billiards in an ellipse, and involves developing the action-angle coordinates for this system, and exploiting it to construct a Riesz basis for $L^2$.
As a corollary, we obtain a spectral rigidity result for this setting.
Finally, we study the integrability question in the setting of potentials that are periodic. 
\end{abstract}

\section{Introduction and Main Results}\label{sec:intro}

\subsection{Standard maps and Frenkel-Kontorova model}\label{subsec:FrenKonIntro}





The \textit{Frenkel-Kontorova model} is a standard model originating from condensed matter physics. 
This model describes an equilibrium of a system of 1-dimensional particles that interact according to spring-like nearest neighbors interactions, under the influence of a potential.
This model can be described in terms of a difference equation.
If $\set{x_n\in\R}_{n\in\Z}$ denotes the locations of the particles, then the following relation holds
\begin{equation}
\label{equation:FK_equation}
x_{n+1}-2x_n+x_{n-1} = V'(x_n),
\end{equation}
where $V:\R\to\R$ is the \textit{potential}, and from now on  we assume that it is periodic.
Sequences satisfying \eqref{equation:FK_equation} are the $x$-projection of orbits $(x_n,y_n)_{n\in\Z}$ of the so-called \textit{standard map} $F_V$ with potential $V$
\begin{equation}
    \label{equation:standard_map}
F_V:
\left\{
\begin{array}{rcl}
     x_1 & = & x_0+y_0+V'(x_0) \\
     y_1 & = & y_0+V'(x_0). 
\end{array}
\right.
\end{equation}
    

This is an example of an exact twist map of a cylinder.
In Section \ref{sec:bg} we provide the relevant definitions and recall the relevant results we use.
A famous example of a potential that is chosen in this setting is $V_k(x)=\frac{k}{2\pi}\sin(2\pi x)$.
The corresponding map,  often called the Chirikov-Taylor-Greene standard map, has remarkable physical interpretations, see e.g. \cite{lichtenberg2013regular,AUBRY1983240,RevModPhys.64.795}.
We will not, however, study the potential $V_k$ in this work.
Different works, see \cite{lazutkin_integrable, suris1989integrable}, studied which of these systems are so-called \textit{integrable} in a sense of Liouville. This notion is inspired by Hamiltonian dynamics and Arnold-Liouville Theorem where integrability is related to the existence of first integrals, independent and in involution. 

\subsection{Standard maps of Suris type}\label{subsec:StdSurisIntro}

In the discrete case, Suris \cite{suris1989integrable} studied the existence of first integrals for the system \eqref{equation:FK_equation} when $V$ varies along a $1$-parametric families of potential $(V_{\varepsilon})_{\varepsilon\in I}$. A \textit{first integral} of the system \eqref{equation:FK_equation} is a function $\Phi=\Phi(x,x')$ satisfying for any integer $n$
\[
\Phi(x_{n+1},x_n)=\Phi(x_n,x_{n-1}).
\]
He showed the following result:

\begin{theorem}[Suris \cite{suris1989integrable}]
\label{theorem:Suris}
Let $a>0$, and $(V_{\varepsilon})_{0\leq \varepsilon<a}$ be a one-paremeter family of potentials $V_{\varepsilon}$, analytic in $\varepsilon$.
    Assume that for any $0 \leq \varepsilon < a$ and  $V=V_{\varepsilon}$, the system \eqref{equation:FK_equation} admits a first integral $\Phi_{\varepsilon}$ of the form
    \[
    \Phi_{\varepsilon}=\Phi_0+\varepsilon\Phi_1
    \]
such that $\Phi_{\varepsilon}(x,x')=\Phi_{\varepsilon}(x',x)$ for any $0\leq \varepsilon < a$ and $x,x'\in \R$. Then $V_{\varepsilon}$ is given by
\begin{multline}
    V_{\varepsilon}'(x) = \\
    \frac{2}{\omega}\arctan
\left(
\frac{\frac{\omega\varepsilon}{2}\left(
A\sin(\omega x)+B\cos(\omega x)+C\sin(2\omega x)+D\cos(2\omega x)
\right)}
{1-\frac{\omega\varepsilon}{2}\left(
A\cos(\omega x)-B\sin(\omega x)+C\cos(2\omega x)-D\sin(2\omega x)+E
\right)}
\right)
\end{multline}
and $\Phi_0,\Phi_1$ can be chosen to be
\[
\Phi_0(x,x')=\frac{1}{\omega}\left(1-\cos(\omega(x'-x))\right)
\]
and
\begin{multline}
\Phi_1(x,x')=
\frac{1}{2\omega}\left(
A(\cos\omega x+\cos\omega x')
-B(\sin\omega x+\sin\omega x')\right.\\
\left.
+C\cos\omega(x+x')
-D\sin\omega(x+x')
+E\cos\omega(x'-x)
\right).
\end{multline}
\end{theorem}

\begin{remark}
    In \cite{suris1989integrable}, the result of Suris studies the case when $V$ is not necessarily periodic, which gives rise to two other classes of potentials for which the system \eqref{equation:FK_equation} admits a first integral. We do not give the explicit formulae of the potentials in these classes here as we will consider the periodic case only.
\end{remark}

Note that, by changing $A$ into $\tfrac{\lambda}{1-\lambda E}A$, $B$ into $\tfrac{\lambda}{1-\lambda E}B$, etc. where $\lambda=\omega\varepsilon/2$,  one can elliminate redundant parameters of $V_{\varepsilon}$, and we introduce the class of \textit{Suris' potentials}:

\begin{definition}
\label{definition:Suris_potential}
A \textit{Suris' potential} is a function $V:\R\mapsto\R$ of the form
\begin{equation}\label{eq:SurisPotentialGeneralForm}
V(x) = \frac{2}{\omega}\int\limits_0^x\arctan
\left(
\frac{A\sin(\omega \xi)+B\cos(\omega \xi)+C\sin(2\omega \xi)+D\cos(2\omega \xi)}{1-A\cos(\omega \xi)+B\sin(\omega \xi)-C\cos(2\omega \xi)+D\sin(2\omega \xi)}
\right)d\xi
\end{equation}where $\omega\in\R$ is the \textit{frequency} of $V$, and $A,B,C,D$ are real constants such that $V$ is defined everywhere. 
For concreteness let us assume that $\omega=2\pi$.
We call 
\[
\varepsilon := \sqrt{A^2+B^2+C^2+D^2}\geq 0
\]
the \textit{eccentricity} of $V$. 
\end{definition}

The notation $\varepsilon$ here is independent of Suris' result \ref{theorem:Suris}, and will refer only to the eccentricity of a potential $V$.
This terminology is of course inspired by ellipses.
In Subsection \ref{subsec:SimilarityToBilliards} we draw parallel lines between the Suris potential and billiards in an ellipse, in terms of their action-angle coordinates, giving some justification for this choice of terminology.
Note that $V$ is well-defined when $\varepsilon$ is sufficiently small, say $\varepsilon\leq \frac{1}{4}$.
Going forward, we always assume that assumption about the eccentricity.
The choice $\omega=2\pi$ implies that the maps we consider are $1$-periodic in $x$.
In this form, the tuple $(A,B,C,D)$ is uniquely determined by the Suris potential.

As Proposition \ref{prop:rotationIntervalIncludesStuff} below shows, Suris maps with small enough eccentricity are \textit{rationally integrable} in the interval $[\frac{1}{6},\frac{1}{3}]$: they posses invariant curves of any rotation number in the interval $[\frac{1}{6},\frac{1}{3}]$ (see Definition \ref{def:ratioInteg} in Section \ref{sec:bg} below).
Our main result shows that this property is locally unique for the Suris potentials.
Namely, we prove the following:
\begin{theorem}\label{thm:LocalRigidity}
   There exists an integer $r\geq 23$, $\varepsilon_*>0$ and $\delta>0$ with the following property: for $0<\varepsilon<\varepsilon_*$, and $F=F_{V_S+W}$ a twist map where $V_S$ is a Suris potential of eccentricity $\varepsilon$, and $\normC{W}{r}\leq \delta$;
   if $F$ is rationally integrable in $[\frac{1}{6},\frac{1}{3}]$, then $V_S+W$ is itself a Suris potential.
\end{theorem}
This result is inspired by similar works about local results about the Birkhoff conjecture, regarding the integrability of billiards.
In particular, in this work we mimic the method used in \cite{avila2016integrable}. 
That method was improved in subsequent works (see \cite{kaloshin2018local,huang2018nearly, koval2021local}).
It was also adapted to symplectic billiards in \cite{tsodikovich2025local}.
In \cite{ARNAUD2023109175}, the authors considered deformations of integrable twist maps (in general dimension).
They showed that a deformation which is linear in the deformation parameter cannot be integrable.
It is important to emphasize that this result considers a linear perturbation of the potential, while Suris' example invovles a linear deformation of the integral, so the results do not contradict.
Our result allows for a more general deformation, but is specific for a deformation of the Suris potential.
In \cite{Lomel__2000}, the authors studied the destruction of the saddle connection of the Suris map under some specific types of deformation.
This implies that integrability breaks under this deformation. 
In that sense, our result generalizes that result, since we show that integrability breaks under arbitrarily small perturbation.
In the non-periodic case, an area preserving diffeomorphism of $\R^2$ with a similar form is the McMillan map.
The rigidity of the integrability of that map was showed in \cite{martin2009exponentially}.

Let us describe one corollary of our main result.
Suppose that one would try to prove by Theorem \ref{theorem:Suris} by considering deformations quadratic in $\varepsilon$, i.e., integrals of the form
\[\Phi_\varepsilon=\Phi_0+\varepsilon\Phi_1+\varepsilon^2\Phi_2,\]
and in principle, one might find new examples of integrable potentials this way.
Our result shows that in this approach one will not find new rationally integrable potentials.
This does not rule out the possibility of finding new potentials in which some of the invariant curves become singular.

\subsection{Spectral rigidity}

Given a potential $V$, the map $F_V$ is a so-called exact area-preserving twist map of the cylinder -- see Section \ref{sec:bg} for more details. It comes with a generating map
\[
H_V(x_1,x_2) = \frac{1}{2}(x_2-x_1)^2 + V(x_1),
\]
defined in such a way that orbits of $F_V$ corresponds to configurations \((x_i)_{i\in\mathbb{Z}}\) satisfying
\[
\partial_{2}H_V(x_{i-1},x_i) + \partial_{1}H_V(x_i,x_{i+1})\big) = 0.
\]
Here $\partial_1$ and $\partial_2$ denote the derivatives with respect to the first and second variables.
Given an orbit of $F_V$ corresponding to a $q$-periodic configuration \(x=(x_i)_{i\in\mathbb{Z}}\), its action is the quantity 
\[
A_V(x) = \sum_{i=0}^{q-1} H_V(x_i,x_{i+1}).
\]
The closure of the set of all such action values, taken over periodic orbits forms the \emph{action spectrum} of $V$ and will be denoted by $\mathscr A(V)$.

For each rotation number \(\omega \in \mathbb{R}\), one can define the average action of minimizing orbits of rotation number $\omega$ of $F_V$, and denote it by $\beta_V(\omega)$. See \cite{Gole, Siburg} for the precise definition.
The function $\beta_V:\R\to\R$ is a convex function called Mather's beta function and it encodes the spectral data of the map $F_V$.

One can ask what are the potentials with the same action spectrum, respectively the same beta function.
Given a potential $V$, the set of potentials with the same action spectrum, namely the set
\[
\mathscr I(V) = \{ W\,|\, \mathscr A(W)= \mathscr A(V) \}
\]
is called the \textit{isospectral set} of $V$. An important problem is to classify these isospectral sets.

One can consider the analogous problem with Mather's beta function by considering, for a given potential $V$, the set 
\[
\mathscr I_{\beta}(V) = \{ W\,|\, \beta_{W}=\beta_{V}\}
\]
called the \textit{marked isospectral set} of $V$. The terminology comes from the interpretation of $\beta$ as a \textit{marked} version of the action spectrum: $\beta$ provides the rotation numbers associated to the considered spectral data.

We can deduce from Theorem \ref{thm:LocalRigidity} the first following result:

\begin{corollary}\label{cor:SpecRigid1}
There exist $r\geq 23$, $\varepsilon_*>0$ and $\delta>0$ with the following property: given a Suris potential $V_S$ of eccentricity $\varepsilon\in(0,\varepsilon_*)$, any other potential $V$ which is $\delta$-$C^{r}$ close to $V$ and satisfying
\[
\beta_V=\beta_{V_S}
\]
is itself a Suris potential.
\end{corollary}

One can also ask the question in terms of \textit{deformations}: given an initial potential $V_0$, is
it possible to deform it into a smooth family $V_{\tau}$ of potentials parametrized by a real variable $\tau$ such that the action spectrum remains constant along the deformation?

\begin{corollary}\label{cor:SpecRigid2}
There exist $r\geq 23$, $\varepsilon_*>0$ and $\delta>0$ with the following property: given a Suris potential $V_S$ of eccentricity $\varepsilon\in(0,\varepsilon_*)$, any $C^{23}$-smooth one-parameter family of potentials $(V_{\tau})_{\tau\in I}$defined on an interval $I$ containing $0$ such that $V_0 = V_S$, and satisfying

\[
\normC{V_\tau-V_S}{r}<\delta,\quad 
\mathscr A(V_{\tau}) =\mathscr A(V_{S}) \qquad \tau\in I
\]
consists entirely of Suris potentials.
\end{corollary}

The proofs of Corollaries \ref{cor:SpecRigid1} and \ref{cor:SpecRigid2} will be presented in Section \ref{section:rigidity_cor_proof}.

\subsection{Periodic rigidity}

We will also focus on some other sense in which Suris' potential is rigid. 
The Suris potential can definitely be $\frac{1}{2}$-periodic, if we set $A=B=0$.
Therefore there are $\frac{1}{2}$-periodic integrable potentials.
However, the resulting map $F_V$ does not have an invariant curve of $2$ periodic orbits. 
This can be checked by a direct computation.
The next result, which is in the spirit of \cite{10.1093/imrn/rnab366}, says that the same holds for higher periods as well.
In particular, there is no hope of finding non-constant rationally integrable potentials with period smaller than $\frac{1}{2}$.
\begin{theorem}\label{thm:NoHigherOrderInteg}
    Suppose that $V$ is a smooth function and $\frac{r}{k}$-periodic, with $r$ being coprime with $k$, and $k\geq 2$.
    If $F_V$ has an invariant curve of $k$-periodic orbits, then $V$ is constant.
\end{theorem}

\textbf{Structure of the paper:} In Section \ref{sec:bg} we recall basic properties about twist maps that are used in the paper. 
In Section \ref{sec:ActionAngle} we investigate the action-angle coordinates for the Suris map, as well as its rotation interval.
This is used in Section \ref{sec:FourierBasis} to construct a special Riesz basis of $L^2[0,1]$ associated to a given Suris potential.
We use this basis in Section \ref{sec:actionEstimate} to estimate the action deviations betweens the Suris potential and its perturbation.
This then allows us to complete the proof of Theorem \ref{thm:LocalRigidity} in Section \ref{sec:ProofOfMainThm}.
In Section \ref{section:rigidity_cor_proof} we prove Corollaries \ref{cor:SpecRigid1} and \ref{cor:SpecRigid2}.
In Section \ref{sec:rotSymmetry} we give the proof of Theorem \ref{thm:NoHigherOrderInteg}.

\subsection*{Acknowledgmenets}
D.T. is grateful to Vadim Kaloshin and Illya Koval for useful discussion.
C.F. acknowledges the support of the Italian Ministry of University and Research’s PRIN
2022 grant “Stability in Hamiltonian dynamics and beyond”, as well as
the Department of Excellence grant MatMod@TOV (2023-27) awarded
to the Department of Mathematics of University of Rome Tor Vergata.
D.T. is supported by ERC Grant \#885707.

\section{Twist map of cylinders}\label{sec:bg}
In this section we recall the basic properties of twist map we use in this work. 
For more detailed introduction, see e.g., 
\cite{Bangert1988MatherSF,gole2001symplectic,Siburg}.
The map $F_V$ of \eqref{equation:standard_map} is an example of an exact twist map.
A diffeomorphism of the cylinder $T(q,p)=(Q(q,p),P(q,p))$ is called an exact twist map if it satisfies:
\begin{enumerate}
    \item \label{itm:twist} Twist condition: if we consider the lift of $Q$ to  $\R$ then $\frac{\partial Q}{\partial p}\neq 0$.
    \item \label{itm:exact} There exists a function $H:\R^2\to\R$, called a \textit{generating function}, which is $1$ periodic in both variables and satisfies
    \[T(q,p)=(Q,P)\Longleftrightarrow\begin{cases}
        p=-\partial_1H(q,Q),\\
        P=\partial_2H(q,Q),
    \end{cases}\]
    where $\partial_1,\partial_2$ denote partial derivatives with respect to the first and second variables.
\end{enumerate}
Note that Condition \ref{itm:exact} has an equivalent formulation in terms of $1$-form, namely
\[dH=PdQ-pdq,\]
from which follows that $T$ is area preserving.

In addition, condition \ref{itm:exact} allows to analyze orbits using a variational approach.
For example, orbits are exactly critical points of the formal action functional,
\begin{equation}\label{eq:ActionDef}
    A(\set{x_i})=\sum\limits_{i=-\infty}^\infty H(x_i,x_{i+1}).
\end{equation}
For periodic orbits (where there is some $N>0$ such that $x_{i+N}-x_i$ is an integer), we have an actual well defined action, obtained by summing the generating function along one orbit.
The closure of the set of all possible actions of all possible periodic orbits for a twist map $T$ is called \emph{the spectrum of $T$},$\mathscr A(V)$.
In the case of the map $F_V$ above, one can check that the generating function is given by 
\[H_V(x_1,x_2)=\frac{1}{2}(x_2-x_1)^2+V(x_1).\]

In this work, we are interested in invariant curves of twist maps.
These are non-contractible loops on the cylinder which are preserved by $T$.
By a theorem of Birkhoff, all invariant curves are graphs of functions.
By restricting $T$ to an invariant curve, we get a circle diffeomoprhism, and hence, has a well-defined rotation number.
If the rotation number is rational, and the curve consists of periodic points, then all these orbits share the same action.
We will be interested in ``abundance" of invariant curves with rational rotation numbers:

\begin{definition}\label{def:ratioInteg}
    A twist map $T$ is called \emph{rationally integrable} in the interval $[a,b]$, if for all $\rho\in[a,b]\cap\Q$, $T$ has an invariant curve of rotation number $\rho$ consisting of periodic orbits.
\end{definition}

We also recall the construction of the Mather $\beta$-function associated to an exact twist map $T$, see \cite{Gole, Siburg}.
Given a rotation number $\rho$, we consider a minimal orbit that has rotation number $\rho$, $\set{x_k(\rho)}_{k\in\Z}$, and then the $\beta$-function is defined by
\[\beta(\rho)=\lim_{N\to\infty}\frac{1}{2N}\sum\limits_{k=-N}^{N-1}H(x_k(\rho),x_{k+1}(\rho)).\]

In the case of $\rho=\frac{m}{n}\in\Q$, the above simplifies to be
\[\beta(\frac{m}{n})=\frac{1}{n}\sum\limits_{k=0}^{n-1}H(x_k,x_{k+1}).\]
The property of the $\beta$-function which is relevant to us, is the following:

\begin{proposition}[\cite{Siburg}, Theorem 1.3.7]
\label{proposition:differentiability_beta}
    Let $\varrho=\frac{p}{q}\in\Q$. The function $\beta$ is differentiable at $\varrho$ is and only if it contains an invariant curve consisting of periodic points of rotation number $p/q$. 
\end{proposition}

\section{Action-angle coordinates for Suris map}\label{sec:ActionAngle}
Our goal in this section, is to derive formulae for the action angle coordinates for the standard map with Suris' potential.
This follows the classical recipe given in \cite{arnold1989mathematical}.

\subsection{From Suris integrability to invariant curves}

Let us give first an expression of the integral of motion adapted to Suris potentials $V$ introduced in Definition \ref{definition:Suris_potential}.
Fix $A,B,C,D\in\R$, and consider the corresponding Suris potential $V$. 
Suris introduced the quantity $\Phi(x,x')=\Phi_{A,B,C,D}(x,x')$ defined for any $x,x'\in\R$ by
\begin{multline}
\label{equation:first_integral_I}
\Phi(x,x') = -\cos(2\pi (x'-x))+A(\cos(2\pi x')+\cos(2\pi x))-B(\sin(2\pi x')+\sin(2\pi x))+\\
+C\cos(2\pi (x+x'))-D\sin(2\pi(x+x')).
\end{multline}
He showed \cite{suris1989integrable} the following Proposition:
\begin{proposition}[Suris]
\label{proposition:Suris_invariant}
Let $V$ be a Suris' potential of parameters $A,B,C,D\in\R$ and $(x_n,y_n)_{n\in\Z}$ be an orbit of the standard map $F_V$ of potential $V$ of the form \eqref{eq:SurisPotentialGeneralForm}.
Then for any integer $n$
\[\Phi(x_n,x_{n+1}) = \Phi(x_{n+1},x_{n+2}).\]
\end{proposition}

We seek now to express $\Phi$ in terms of $(x,y)$-coordinates on the cylinder, and we define the quantity $I(x,y)= I_{A,B,C,D}(x,y)$ by
\begin{multline}
\label{equation:first_integral_I_2}
I(x,y) = -\cos(2\pi y)+A(\cos(2\pi x)+\cos(2\pi (x-y)))-B(\sin(2\pi x)+\sin(2\pi (x-y)))+\\
+C\cos(2\pi (2x-y))-D\sin(2\pi(2x-y)).
\end{multline}
In the notations of Proposition \ref{proposition:Suris_invariant}, since $x_n = x_{n+1}-y_{n+1}$, the quantities $I$ and $\Phi$ satisfy the relation
\[
\Phi(x_n,x_{n+1}) = I(x_{n+1},y_{n+1}),\qquad n\in\Z.
\]
Hence $I$ is a first integral for the system given by $F_V$: $I\circ F_V = I$.
We are therefore brought to consider the set $\mathcal I(V)$ of values reached by $I(x,y)$ as $x,y$ vary in $\R$.

\begin{lemma}
    \label{lemma:intervalInvariant}
    The set $\mathcal I(V)$ is an interval of the form
    \[
    \mathcal I(V) = [I^-_{A,B,C,D},I^+_{A,B,C,D}]
    \]
    where $I^-_{A,B,C,D}\to -1$ and $I^+_{A,B,C,D}\to 1$ as $\varepsilon=\sqrt{A^2+B^2+C^2+D^2}\to0$.
\end{lemma}

\begin{proof}
    Since $I$ is continuous and $1$-periodic in $x$ and in $y$, the set $\mathcal I(V)$ is a compact connected set of $\R$ as the image of $[0,1]\times[0,1]$ by $I$. Hence it is an interval of the form $\mathcal I(V) = [I^-_{A,B,C,D},I^+_{A,B,C,D}]$. Let us estimate $I^{\pm}_{A,B,C,D}$ when $\varepsilon\to 0$. 
    
    Using classical trigonometric identities, $I(x,y)$ can be expressed as
    \begin{equation}
        \label{equation:expansion_I}
        I(x,y) = -\alpha(x)\cos(2\pi y)+\beta(x)\sin(2\pi y)+\gamma(x)
    \end{equation}
    where 
    \[
    \alpha(x) = 1-A\cos(2\pi x)+B\sin(2\pi x)-C\cos(4\pi x)+D\sin(4\pi x)=1+\mathcal O(\varepsilon),
    \]
    \[
    \beta(x) = A\sin(2\pi x)+B\cos(2\pi x)+C\sin(4\pi x)+D\sin(4\pi x) = \mathcal O(\varepsilon)
    \]
    and
    \[
    \gamma(x) = A\cos(2\pi x)-B\sin(2\pi x)=\mathcal O(\varepsilon)
    \]
    and the asymptotics are uniform in $x$.
    The maximum of $I(x,y)$ for a fixed $x$ is therefore given by
    \[
    \max_{y\in\R} I(x,y) = \sqrt{\alpha(x)^2+\beta(x)^2}+\gamma(x) = 1+\mathcal O(\varepsilon)
    \]
    and the asymptotics are uniform in $x$.
    Hence $I^+_{A,B,C,D}=\max I = 1+\mathcal O(\varepsilon)$. Similarly, $I^-_{A,B,C,D}= -1+\mathcal O(\varepsilon)$ and the result follows.
\end{proof}

Given a Suris' potential $V$ of parameters $A,B,C,D\in\R$, let $\eta\in \mathcal{I}(V)$ and $(x,y)$ for which $I(x,y)=\eta$.
Note that an orbit $(x_n,y_n)=F_V^n(x,y)$  stays in the level set $\{I=\eta\}$, in the sense that it satisfies
\[
I(x_n,y_n)=\eta,\qquad n\in\Z.
\]
Let us \textit{reverse} this expression and expresses $y$ in terms of $x$ and $\eta$.
\begin{lemma}
\label{lemma:graph_invariant_circles}
    In the case when $\varepsilon<1/2$, there exists an integer $k\in\Z$ and $\sigma\in\set{\pm1}$ such that $y$ has the form
    \begin{equation}\label{eq:FormulaInvCurve}
    y = \frac{\sigma}{2\pi}\arccos\left(
    \frac{\gamma(x)-\eta}{\mathcal D(x)}
    \right)+\frac{1}{2}V'(x)+k  
    \end{equation}
    where
    \[
    \qquad \mathcal D(x) = \sqrt{\alpha(x)^2+\beta(x)^2},
    \]
    \[
    \alpha(x) = 1-A\cos(2\pi x)+B\sin(2\pi x)-C\cos(4\pi x)+D\sin(4\pi x),
    \]
    \[
    \beta(x) = A\sin(2\pi x)+B\cos(2\pi x)+C\sin(4\pi x)+D\cos(4\pi x),
    \]
    \[
    \gamma(x) = A\cos(2\pi x)-B\sin(2\pi x).
    \]
\end{lemma}

\begin{proof}
Using the expansion \eqref{equation:expansion_I} of $I(x,y)$, we have the equality
\[
\eta = -\alpha(x)\cos(2\pi y)+\beta(x)\sin(2\pi y)+\gamma(x)
\]
whence
\begin{equation}
    \label{equation:first_integral_intermediate}
    \frac{\alpha(x)}{\mathcal D(x)}\cos(2\pi y)-\frac{\beta(x)}{\mathcal D(x)}\sin(2\pi y)
    = \frac{\gamma(x)-\eta}{\mathcal D(x)}.
\end{equation}
Note that by Cauchy-Schwarz identity, $\alpha(x)\geq 1-2\varepsilon>0$ by assumption on $\varepsilon$. Therefore there exists $\varphi=\varphi(x)\in\left(-\frac{1}{4},\frac{1}{4}\right)$ such that
    \[
    \cos(2\pi\varphi(x)) = \frac{\alpha(x)}{\mathcal D(x)}
    \quad\text{and}\quad
    \sin(2\pi\varphi(x)) = \frac{\beta(x)}{\mathcal D(x)}.
    \]
    Using $\varphi(x)$, Equation \eqref{equation:first_integral_intermediate} can be rewritten as
    \[
    \cos(2\pi(y-\varphi(x))) = \frac{\gamma(x)-\eta}{\mathcal D(x)}.
    \]
    Now $\varphi(x)$ can be computed using the trigonometric relation $\tan(2\pi\varphi(x)) = \frac{\beta(x)}{\alpha(x)}$ which implies\footnote{Here $\tan$ and  $\arctan$ are mutual inverses since $\varphi(x)$ is in the right interval.}
    \[
    \varphi(x) = \frac{1}{2\pi}\arctan\left(\frac{\beta(x)}{\alpha(x)}\right)=\frac{1}{2}V'(x),
    \]
    by the definition of $V$ given in Equation \eqref{eq:SurisPotentialGeneralForm}. The result follows.
\end{proof}
Note that $\sigma$, $k$ from Lemma  \ref{lemma:graph_invariant_circles} depend on $(x,y)$, and might, a priori, be different for $T(x,y)$.
The next lemma ensures us that this is not the case.
Namely, $k,\sigma$ remain constant along the orbits of $T$, provided that $\varepsilon$ is sufficiently small.
This way we can describe a large set of invariant curves of $F_V$.
Given $\eta\in\mathcal I(V)$, $\sigma\in\{\pm1\}$ and $k\in\Z$, based on Equation \eqref{eq:FormulaInvCurve} we define the function
\[
\psi_{\eta}^{\sigma,k}(x) = \frac{\sigma}{2\pi}\arccos\left(
    \frac{\gamma(x)-\eta}{\mathcal D(x)}
    \right)+\frac{1}{2}V'(x)+k
\]
and its graph by 
\[
M_{\eta}^{\sigma,k} = \left\{(x,y)\in S^1\times\R\,:\,y = \psi_{\eta}^{\sigma,k}(x)\right\}.
\]

\begin{proposition}
    \label{proposition:invariantCurves}
    Given $\delta>0$, there is $\varepsilon_0>0$ such that for any $\varepsilon\in(0,\varepsilon_0)$, the graph $M_{\eta}^{\sigma,k}$ is invariant by $F_V$, for any $\eta\in\mathcal I(V)\cap(-1+\delta,1-\delta)$ and any $(\sigma,k)\in\{\pm1\}\times\Z$.
\end{proposition}

\begin{proof}
    Let $(\sigma_0,k_0)\in\{\pm1\}\times\Z$ and $(x_0,y_0)\in M_{\eta}^{\sigma_0,k_0}$. Consider $(x_1,y_1)=F_V(x_0,y_0)$. By Proposition \ref{lemma:graph_invariant_circles}, there is $(\sigma_1,k_1)\in\{\pm1\}\times\Z$ such that $(x_1,y_1)\in M_{\eta}^{\sigma_1,k_1}$.
    Let us show that $(\sigma_0,k_0)=(\sigma_1,k_1)$ for $\varepsilon$ sufficiently small.

    Note first by definition of $\psi_{\eta}^{\sigma_0,k_0}(x)$
    \[
    y_0-\frac{1}{2}V'(x_0)=k_0+\frac{\sigma_0}{2\pi}\arccos(-\eta)+\mathcal O(\varepsilon)
    \]
    uniformly in $x_0$. Since $y_1=y_0+V'(x_0)$, we also have (using the fact that $|V'|=\mathcal{O}(\varepsilon)$ uniformly in $x$)
    \[
    y_1-\frac{1}{2}V'(x_1)=k_0+\frac{\sigma_0}{2\pi}\arccos(-\eta)+\mathcal O(\varepsilon)
    \]
    uniformly in $x_0$, and therefore
    \begin{equation}
    \label{equation:cell_arccos}
    k_1+\frac{\sigma_1}{2\pi}\arccos(-\eta)=k_0+\frac{\sigma_0}{2\pi}\arccos(-\eta)+R_{\varepsilon}(x_0)
    \end{equation}
    where $R_{\varepsilon}(x_0) = \mathcal O(\varepsilon)$ uniformly in $x_0$.
    
    Let $\delta>0$. The image of the interval $(-1+\delta,1-\delta)$ by a function of the form $\eta\mapsto k_1-k_0\pm\frac{1}{\pi}\arccos(-\eta)$ is contained in an interval of the form
    \[
    J = (k+r,k+1-r)
    \]
    where $k$ is an integer and $r\in(0,1)$ depends only on $\delta$. Let us choose $\varepsilon_0>0$ such that $|R_{\varepsilon}(x)|<r$ for any $\varepsilon\in(0,\varepsilon_0)$ and any $x_0\in\R$.
    For this choice of $\varepsilon$, Equation \eqref{equation:cell_arccos} implies that $\sigma_0=\sigma_1$ -- otherwise $R_{\varepsilon}(x)$ would belong to an interval of the form given by $J$, which is not possible. Hence
    \[
    k_1-k_0 = R_{\varepsilon}(x_0)\in(-1,1)\cap\Z
    \]
    and therefore $k_0=k_1$.
    This implies that $(x_1,y_1)\in M_{\eta}^{\sigma_1,k_1}=M_{\eta}^{\sigma_0,k_0}$.
\end{proof}

\begin{proposition}\label{proposition:rotationNumbers}
    Given $r>0$, there exists $\varepsilon_*>0$ such that for any $\varepsilon\in(0,\varepsilon_*)$ the map $F_V$ associated to a Suris potential $V$ of eccentricity $\varepsilon$ has invariant curves of any rotation number in any set of the form $(k-\frac{1}{2}+r,k-r)\cup(k+r,k+\frac{1}{2}-r)$ where $k\in\Z$.
\end{proposition}

\begin{proof}
   Given $\delta>0$, consider $\varepsilon_0>0$ given by Proposition \ref{proposition:invariantCurves} and $V$ a Suris potential of eccentricity $\varepsilon\in(0,\varepsilon_0)$. If $\eta\in(-1+\delta,1-\delta)$ and $(\sigma,k)\in\{\pm 1\}\times\Z$, the set $M_{\eta}^{\sigma,k}$ is an invariant curve of $F_V$, as it follows from the same Proposition.
    
    Moreover, since $x_1=x_0+y_0+V'(x_0)$, we can state
    \[
    x_1-x_0 = k+\frac{\sigma}{2\pi}\arccos(-\eta) + \mathcal O(\varepsilon),
    \qquad (x_0,y_0)\in M_{\eta}^{\sigma,k}
    \]
    uniformly in $x_0$. If $(x_n,y_n)_{n\in\Z}$ denotes an orbit in $M_{\eta}^{\sigma,k}$, we deduce that
    \[
    \frac{x_n-x_0}{n} = k+\frac{\sigma}{2\pi}\arccos(-\eta) +\mathcal O(\varepsilon),\qquad n\in\Z
    \]
    uniformly in $n$. It follows that the rotation number $\omega = \omega(\eta,\sigma,k)$ of $F_V$ on $M_{\eta}^{\sigma,k}$ satisfies the estimates
    \begin{equation}
    \label{equation:rotation_number}
    \omega = k+\frac{\sigma}{2\pi}\arccos(-\eta) +R_{\varepsilon}
    \end{equation}
    where $|R_{\varepsilon}|\leq C\varepsilon$ for any $\varepsilon>0$ and $C>0$ is an independent constant. 

    Now since $(M_{\eta}^{\sigma,k})_{\eta\in (-1+\delta,1-\delta)}$ is a local foliation of an open set of $\mathbb S^1\times\R$ and $F_V$ is a twist map, the set 
    \[
    J_{\sigma,k} = \{\omega(\eta,\sigma,k)\,|\, \eta\in (-1+\delta,1-\delta)\}
    \]
    is an interval of the form $(\eta_-,\eta_+)$ where $\eta_-<\eta_+$. Considering Equation \eqref{equation:rotation_number} when $\eta\to-1+\delta$ and $\eta\to 1-\delta$, we obtain
    \begin{equation}
    \label{equation:eta_moins}
    \eta_- \leq k+\frac{\sigma}{2\pi}\arccos(1-\delta)+C\varepsilon
    \end{equation}
    and
    \begin{equation}
    \label{equation:eta_plus}
    \eta_+\geq k+\frac{\sigma}{2\pi}\arccos(-1+\delta)-C\varepsilon.
    \end{equation}
    Hence given $r>0$, 
    consider $\delta>0$ such that
    \[
    \frac{1}{2\pi}\arccos(1-\delta)<r.
    \]
    Consider the corresponding $\varepsilon_0>0$ given by Proposition 
    \ref{proposition:invariantCurves}, and fix $\varepsilon_*\in(0,\varepsilon_0)$ such that 
    \[
    \frac{1}{2\pi}\arccos(1-\delta)+C\varepsilon_*<r.
    \]

    Consider now a Suris potential of eccentricity $\varepsilon\in(0,\varepsilon_*)$, and assume first that $\sigma=1$. The corresponding values of $\eta_-$ and $\eta_+$ satisfy, according to \eqref{equation:eta_moins} and \eqref{equation:eta_plus}, the following inequalities:
    \[
    \eta_-\leq k+\frac{1}{2\pi}\arccos(1-\delta)+C\varepsilon\leq k+r
    \]
    and
    \begin{multline}
     \eta_+\geq k+\frac{1}{2\pi}\arccos(-1+\delta)-C\varepsilon=
     k+\frac{1}{2}-\frac{1}{2\pi}\arccos(1-\delta)-C\varepsilon\\
     \geq k+\frac{1}{2}-r,
    \end{multline}
    hence the set $J_{\sigma,k}$ contains the interval $(k+r,k+\frac{1}{2}-r)$. In the case when $\sigma=-1$, the argument shows that $J_{\sigma,k}$ contains the interval $(k-\frac{1}{2}+r, k-r)$, and the proof follows.
\end{proof}
By setting $k=0$ and $r=\frac{1}{6}$, we get our desired rotation interval:
\begin{corollary2}
    
\label{prop:rotationIntervalIncludesStuff}
    For $\varepsilon$ small enough, the map $F_V$ associated to a Suris potential $V$ of eccentricity $\varepsilon$ is rationally integrable in the interval $[\frac{1}{6},\frac{1}{3}]$.
 \end{corollary2}

\subsection{Action angle coordinates}

Proposition \ref{proposition:invariantCurves} provides a full description of the invariant curves for the Suris map (at least for low eccentricities).
Now we use the  the classical recipe \cite[Chapter 10] {arnold1989mathematical} to get action angle coordinates for $T$.
Some caution needs to be taken here, as the setting of \cite{arnold1989mathematical} is that of Hamiltonian flows, while here we discuss twist maps.
In Proposition \ref{prop:GotActionAngle} we justify that this recipe gives us, in fact, action angle coordinates. 
Therefore, we define
\begin{equation}
\label{equation:Arnold}
\Omega(I) = \int_{M_I}ydx,
\qquad\text{and}\qquad 
\theta(I,z) = \frac{\partial S}{\partial I}(I,z)
\end{equation}
with $S(I,z)=\int_{\gamma}ydx$, where $\gamma$ is any path connecting a point $z_0=(x_0,y_0)\in M_I:=M_{I(x_0,y_0)}^{+1,0}$ fixed in advance to another point $z\in M_I$. 

Lemma \ref{lemma:graph_invariant_circles} implies that $M_I$ is a graph over $x\in\R/\Z$ and \eqref{equation:Arnold} translates into
\begin{equation}\label{eq:ActionAngleFormula}
\begin{cases}
    \Omega(x,y) = \frac{1}{2\pi}\int_0^{1}\arccos\left(
    \frac{A\cos(2\pi x)-B\sin(2\pi x)-I(x,y)}{\mathcal D(x)}
    \right)dx,\\
\theta(x,y) = \Theta(y)\int_0^x\frac{d\tau}{\sqrt{\mathcal D(\tau)^2-(I(x,y)-A\cos(2\pi \tau)+B\sin(2\pi \tau))^2}},
\end{cases}
\end{equation}
where $\Theta(y)$ is a normalization function such that $\theta(1,y)=1$ for all $y$ (i.e., it is the reciprocal of the integral that appears in the formula for $\theta(x,y)$ in the interval $[0,1]$).

While the formulae themselves seem rather complicated, what is important for us is that they're analytic in the parameter $I(x,y)$.
Moreover, in Subsection \ref{subsec:SimilarityToBilliards} we highlight a special case where these formulae simplify.

\begin{proposition}\label{prop:GotActionAngle}\cite[Chapter 10]{arnold1989mathematical}
    In the coordinates $(\theta,\Omega)$, the standard map $F_V$ writes 
    \[
    \left\{
    \begin{array}{rcl}
         \theta_1 & = & \theta_0+g(\Omega_0), \\
         \Omega_1 & = & \Omega_0, 
    \end{array}\right.
    \]
     where $g$ is some smooth function.
\end{proposition}

\begin{proof}
Given a standard map with Suris potential $V$, consider $I$ the conserved quantity of $F_V$, (see Proposition \ref{proposition:Suris_invariant}).
We view $I$ as a Hamiltonian on the phase cylinder of $F_V$.
Denote by $\varphi_t$ the Hamiltonian flow.
Then Proposition \ref{proposition:Suris_invariant}, implies that for all $(x,y)$ on the cylinder there is $t(x,y)>0$ such that
\[\varphi_{t(x,y)}(x,y)=T(x,y).\]
For the Hamiltonian flow $\varphi_t$, we know that the change of variables $\eqref{equation:Arnold}$ conjugates $\varphi_t$ to action angle coordinates.
Namely, if $\psi$ denotes the change of coordaintes in $\eqref{equation:Arnold}$, then $\psi$ is symplectic and
\[\psi\circ\varphi_t\circ\psi\inv (\theta,\Omega) = (\theta+t\Omega,\Omega).\]
This formula holds for all $t$, so in particular we may plug in it $t=t\circ\psi\inv(\theta,\Omega)$. 
The result is that
\[\psi\circ T\circ \psi\inv (\theta,\Omega) = (\theta+t\circ \psi\inv (\theta,\Omega)\Omega,\Omega).\]
But since left-hand side is a symplectic map, then so must be the right hand side, and so the derivative of the first component with respect to $\theta$ must be $1$. 
This means that the summand $t\circ \psi\inv (\theta,\Omega)$ does not actually depend on $\theta$, and we got the desired form.
\end{proof}

\subsection{An interesting special case}\label{subsec:SimilarityToBilliards}
We would like to point that in the special case where $A=B=0$ the formulae in \eqref{eq:ActionAngleFormula} can be greatly simplified.
In this case we may choose $D=0$ (by applying suitable change of variables of the form $x\mapsto x+\alpha$), and then eccentricity is just $\varepsilon=|C|$.
Moreover, in this case $V'$ is odd, so by conjugating $F_V$ with the map $(x,y)\mapsto (-x,-y)$, we may also assume that $C<0$, so $C=-\varepsilon$.
Then the function $\mathcal{D}$ from Lemma \ref{lemma:graph_invariant_circles} simplifies to
\[\mathcal{D}(x)=\sqrt{1+\varepsilon^2+2\varepsilon\cos(4\pi x)}.\]
We also denote
\[k(x,y)=\sqrt{\frac{4\varepsilon}{(1+\varepsilon)^2-I(x,y)^2}}.\]
First we simplify the integral that appears in the formula for $\theta(x,y)$:
\[D(\tau)^2-I(x,y)^2=1+\varepsilon^2+2\varepsilon\cos(4\pi\tau)-I(x,y)^2=(1+\varepsilon)^2-I(x,y)^2-4\varepsilon\sin^22\pi\tau,\]
so
\[\frac{\theta(x,y)}{\theta(1,y)}=\int\limits_0^x\frac{d\tau}{\sqrt{1-k(x,y)^2\sin^2(2\pi\tau)}}=F(2\pi x,k(x,y)),\]
where 
\[F(\varphi,k)=\int\limits_0^\varphi\frac{d\varphi}{\sqrt{1-k^2\sin^2\varphi}}\]
is the incomplete elliptic integral of the fisrt kind.
Thus our normalization for $\theta$ implies that
\begin{equation}\label{eq:AngleCoordSpecial}
    \theta(x,y)=\frac{F(2\pi x,k(x,y))}{4K(k(x,y))},
\end{equation}
where $K(k)=F(\frac{\pi}{2},k)$ is the complete elliptic integral of the first kind.
Note that the dependence of $\theta(x,y)$ on $y$ is only via the value of the integral $I(x,y)$.
Using this we can get an expression for the rotation number on a given invariant curve.
This is just $\theta(F_V(x,y(x)))-\theta(x,y(x))$, where $y(x)$ is given by \eqref{eq:FormulaInvCurve}.
In particular, this should not depend on $x$, so we can compute it for $x=0$, in which case $y(0)=\frac{1}{2\pi}\arccos\frac{-I(x,y)}{\mathcal{D}(x)}$, $V'(0)=0$, and $\theta = 0$, $\mathcal{D}(0)=1+\varepsilon$, and then we get
\begin{equation}\label{eq:rotNumberSpecial}
    \omega = \frac{F(\arccos\frac{-I(x,y)}{1+\varepsilon},k(x,y))}{4K(k(x,y))}
\end{equation}
The formulae \eqref{eq:AngleCoordSpecial}, \eqref{eq:rotNumberSpecial} share remarkable sembelence to the action-angle coordinates of the billiard system inside an ellipse, see, e.g. \cite{chang1988elliptical, kaloshin2018local}.
The main difference is that the arcsine function is replaced with arccosine.

\subsection{The invariant curve of action $\frac{1}{4}$}

It follows from Proposition \ref{prop:rotationIntervalIncludesStuff} that if $\varepsilon$ is sufficiently small, $F_V$ has an invariant curve of rotation number $1/4$. 
The latter will play a crucial role in the proof because of some of its particular properties that will be now discussed.

Consider the inverse of the action-angle coordinate map $\psi$ from Proposition \ref{prop:GotActionAngle}. We can write $x$ and $y$ in terms of $\theta$ and $\Omega$. For each $\Omega$ in some interval, $x_\Omega(\theta)$ is a diffeomorphism of $S^1$ on itself whose inverse is denoted by $\theta_{\Omega}(x)=\theta(\Omega,x)$.

\begin{proposition}
\label{proposition:expansion_theta}
    The map $\theta(\Omega,x)$ admits the following expansion
    \[
     \theta(\Omega,x) = \theta_{\frac{1}{4}}(x)+\left(\Omega-\frac{1}{4}\right)u(x)+v(\Omega,x),
    \]
    where 
    \begin{enumerate}
        \item\label{itm:theta14AtZero} $\theta_{\frac{1}{4}}(x)\to x$ in $C^1$-norm as $\varepsilon\to 0$;
        \item\label{itm:uAsABtoZero} $u\to 0$ in $C^1$-norm as $A,B\to 0$;
        \item\label{itm:BoundsOnV} $|v(\Omega,x)|,|\frac{\partial v}{\partial x}|\leq C(\Omega-\frac{1}{4})^2$ for a given constant $C>0$ independent of $\varepsilon$. Moreover, $v\to 0$ in $C^1$-norm as $\varepsilon\to 0$.
    \end{enumerate}
\end{proposition}

\begin{proof}
    By analyticity of $\theta$, such expansion is well-defined and each term as well as their derivatives are continuous in $\varepsilon$. Hence the boundedness of $v$ in the third item follows immediately. Now when $\varepsilon=0$, $V$ vanishes, hence $\theta_{\frac{1}{4}}(x)=x$, which implies the first and third item. 
    
    For the second item, assume that $A=B=0$. Formulae \eqref{eq:ActionAngleFormula} imply that $\Omega$ is a function of $I$ which admits the following expansion as $I=0$
    \[
    \Omega-\frac{1}{4} =\mu I +\mathcal O(I^3)
    \]
    where $\mu = -\tfrac{1}{2\pi}\int_0^1\frac{dx}{\mathcal D(x)}$. Moreover $\theta_{I} = \theta_{\Omega(I)}$ is given by
    \[
    \theta_{I}(x) = \Theta(I)\int_0^x\frac{d\tau}{\sqrt{\mathcal D(\tau)^2-I^2}} = \theta_{I=0}(x)+\mathcal O(I^2).
    \]
    The last expansion expressed in terms of $\Omega$ implies the second item.
\end{proof}

\section{Deformed Riesz basis for $L^2$}\label{sec:FourierBasis}

Our goal in this section is to construct a Riesz basis for $L^2$, according to which we will do Fourier analysis.
This is similar to the basis constructed in \cite{avila2016integrable}.
The construction of the basis relies on action-angle coordinates.
Using Proposition \ref{proposition:expansion_theta}, we can write
\[\theta(\Omega,x) = \theta\left(x,\frac{1}{4}\right)+\left(\Omega-\frac{1}{4}\right)u(x)+v(x,\Omega),\]
where $u$ and $v$ have the properties shown in Proposition \ref{proposition:expansion_theta}.
We write $\theta_\Omega(x) = \theta(\Omega,x)$.
According to Corollary \ref{prop:rotationIntervalIncludesStuff}, for $\Omega\in[\frac{1}{6},\frac{1}{3}]$ this is a diffeomorphism of $S^1$.

We can use this diffeomorphism to induce a new inner product on the space $L^2[0,1]$:
\begin{equation}\label{eq:NewInnerProduct}
    \langle{f,g}\rangle = \int\limits_0^1 f(x)\overline{g(x)}\theta_{\frac{1}{4}}'(x)dx.
\end{equation}
By a change of variable in the integral, one gets that the collection of functions $(e_q)_{q\in\mathbb Z}$ given by
\[
e_q(x) = e^{2i\pi q\theta_{\frac{1}{4}}(x)}
\]
or equivalently that the collection given by
\[c_q(x) = \cos(2\pi q {\theta_{\frac{1}{4}}(x)}),\,s_q(x)=\sin(2\pi q {\theta_{\frac{1}{4}}(x)}),\]
are orthonormal basis with respect to this inner product.
We also consider, for $q=\pm1,\pm2$,
\[E_q=\frac{e^{2\pi i q\theta_{\frac{1}{4}}(x)}-1}{2\pi iq}.\]
In this case we have the following real and imaginary parts:
\[S_q(x)=\mathrm{Re}E_q(x)=\frac{\sin(2\pi q\theta_{\frac{1}{4}}(x))}{2\pi q},\, C_q(x)=\mathrm{Im}E_q(x)=\frac{1-\cos(2\pi q\theta_{\frac{1}{4}}(x))}{2\pi q}.\]
The collection $\set{E_{\pm 1}, E_{\pm 2},e_q\mid |q|\geq 3}$ is still a basis of $L^2[0,1]$ with respect to our inner product.
While it is no longer orthonormal, we do have
\[\set{E_q\mid q=\pm 1,\pm 2}\perp\set{e_q||q|=0,3,4,...},\] 
with the right hand set being orthonormal.

For $q\geq9$, using division and remainder, we find $p_q\in\N$ and $t_q\in\set{0,1,2,3}$ such that
\[q=4p_q+t_q.\]
Then we set $r_q=\frac{p_q}{q}$.
Then it holds that
\begin{equation}\label{eq:RationalRotNumberCond}
   \left|r_q-\frac{1}{4}\right|q = \frac{t_q}{4},
\end{equation}
from which it follows that $r_q\in[\frac{1}{6},\frac{1}{3}]$.
Furthermore, for $q=3,...,9$ we choose $r_q$ in the following way
\[
\begin{array}{c|ccccccc}
q     & 3 & 4 & 5 & 6 & 7 & 8 \\ \hline
r_q   & \frac{1}{3}  & \frac{1}{4}  & \frac{2}{5}  & \frac{1}{6}  & \frac{2}{7}  &   \frac{2}{8}=\frac{1}{4}
\end{array}
\]
This way $r_q$ is again a rational number with denominator $q$ which is in the interval $[\frac{1}{6},\frac{1}{3}]$.
We can also write it in the form $q=4p_q+t_q$, where $t_q$ is the numerator of $r_q$, and here $|t_q|\in\set{0,1,2,3}$.

We define a deformed basis $(f_q)_q$ which is adapted to our problem (see Section \ref{sec:actionEstimate}):

\begin{definition}\label{def:FourierBasis}
    The collection $(f_q)_{q\in\mathbb Z}$ given for any $q\in\mathbb Z$ and $x\in\mathbb R$ by
    \[
    f_{\pm1}(x) = \pi\partial_BV(x)\pm i\pi\partial_AV(x),
    \]
    \[
    f_{\pm2}(x) = \pi\partial_DV(x)\pm i\pi\partial_CV(x)
    \]
    and if $|q|\geq 3$,
    \[
    f_q(x) = e^{2i\pi q\theta_{r_{|q|}}(x)}\frac{\theta_{r_{|q|}}'(x)}{\theta_{\frac{1}{4}}'(x)}.
    \]
    is called a \emph{deformed Fourier basis}.
\end{definition}
First we make the following observation.
The subspace spanned by $\set{f_{\pm 1},f_{\pm 2}}$ is finite dimensional, so all norms on it are equivalent. 
Moreover, the equivalence constants can be stabilized for small eccentricities.
\begin{lemma}\label{lem:EquivNorm}
    Let $k$ be a positive integer.
    Then there exists a small enough eccentricity $\varepsilon_*>0$ and a constant $M>0$, such that for all  $\varepsilon<\varepsilon_*$, and for all Suris potentials of eccentricity $\varepsilon$, and for all functions $\varphi\in\mathrm{span}_{\C}\set{f_{\pm1},f_{\pm2}}$:
    \begin{equation}\label{eq:NormEquiv}
        \frac{1}{M}\normC{\varphi}{k}\leq \norm{\varphi}_{A,B,C,D}\leq M\normC{\varphi}{k}.
    \end{equation}
    Here $\norm{\cdot}_{A,B,C,D}$ denotes the norm of \eqref{eq:NewInnerProduct} corresponding to a Suris potential with eccentricity $\varepsilon^2=A^2+B^2+C^2+D^2$.
\end{lemma}

\begin{proof}
    
    Using the definition of $\norm{\varphi}_{A,B,C,D}$, we obtain a first bound
    \[
    \norm{\varphi}_{A,B,C,D} \leq M_1\|\varphi\|_{L^2}\leq M_1 \normC{\varphi}{k}
    \]
    where $M_1:=\sup_x |\theta'_{\frac{1}{4}}(x)|\to 1$ as $\varepsilon\to0$.
    
    For the converse bound, consider $\varphi\in\mathrm{span}_{\C}\set{f_{\pm1},f_{\pm2}}$ written as
    \[
    \varphi = a_1f_{1}+a_{-1}f_{-1}+a_2f_{2}+a_{-2}f_{-2},\qquad a_{\pm1},a_{\pm2}\in\C.
    \]
    The vector of coefficients $v = (a_j)_{j\in\{\pm1,\pm2\}}$ can be expressed as
    \[
    v = G(f)^{-1}w,
    \]
    where
    \[
    G(f) := (\langle f_i,f_j\rangle)_{i,j\in\{\pm1,\pm2\}},\quad w:=(\langle\varphi,f_j\rangle)_{j\in\{\pm1,\pm2\}}.
    \]
    Therefore, 
    \begin{multline}
        \label{eq:FirstIneqEquiv}
        \normC{\varphi}{k} \leq \left(\sum_{j\in\{\pm1,\pm2\}}\normC{f_j}{k}\right)\max_{j\in\{\pm1,\pm2\}} |a_j|\leq\\
        \leq \left(\sum_{j\in\{\pm1,\pm2\}}\normC{f_j}{k}\right)\|G(f)^{-1}\|\max_{j\in\{\pm1,\pm2\}} |\langle\varphi,f_j\rangle|
    \end{multline}
    where $\|G(f)^{-1}\|=\max_i\sum_j |G(f)_{i,j}^{-1}|$ is the operator norm of $G(f)^{-1}$ associated to the supremum on vectors' coefficients. By Cauchy-Schwarz inequality
    \begin{equation}
        \label{eq:SecondIneqEquiv}
        \max_{j\in\{\pm1,\pm2\}} |\langle\varphi,f_j\rangle|\leq \norm{\varphi}_{A,B,C,D}\left(\max_{j\in\{\pm1,\pm2\}}\normC{f_j}{k}\right)
    \end{equation}
    Combining \eqref{eq:FirstIneqEquiv} and \eqref{eq:SecondIneqEquiv}, we obtain 
    $\normC{\varphi}{k} \leq M_2\norm{\varphi}_{A,B,C,D}$ where 
    \[
    M_2 =  \left(\max_{j\in\{\pm1,\pm2\}}\normC{f_j}{k}\right)\|G(f)^{-1}\|\max_{j\in\{\pm1,\pm2\}}\normC{f_j}{k}.
    \]
    Due to the analyticity of $V$ in the parameters $A$,$B$,$C$,$D$, it follows, similar to Lemma \ref{lemma:technical_fq_low} below, that 
    \[\normC{f_j-E_j}{k}\to0,\]
    as $\varepsilon\to 0$.
    In addition, since $\theta_{\frac{1}{4}}(x)\to x$ in $C^k$ norm as $\varepsilon\to 0$, we also have
    \[E_j(x)\to\frac{e^{2\pi ijx}-1}{2\pi ij}=:T_j(x).\]
    Hence, for small enough eccentricity:
   \[\normC{f_j}{k}\leq2\normC{T_j}{k}.\]
   Moreover, as $\varepsilon\to0$,
   \[G(f)\to G,\]
   where $G$ is a concrete $4\times 4$ invertible matrix, and hence for small enough eccentricity,
   \[\norm{G(f)^{-1}}\leq 2\norm{G^{-1}}.\]
   This means that $M_2$ is also bounded as $\varepsilon\to0$.
   This finishes the proof.

\end{proof}
We think of $f_{\pm1,\pm2}$ as a basis for the tangent space of the manifold of Suris potentials.
Then, using Taylor expansion, we get the following approximation property:
\begin{proposition}\label{prop:fpm1pm2AreDerivatives}
    Denote by $V(A,B,C,D)$ the Suris potential with those parameters.
    Then there exist $K,K'>0$, such that for all $\alpha,\beta,\gamma,\delta$ small enough, it holds that
    \begin{gather*}
        \|V(A+\alpha,B+\beta,C+\gamma,D+\delta)-V(A,B,C,D)-\frac{\beta+i\alpha }{2\pi} f_1-\frac{\beta-i\alpha}{2\pi} f_{-1}\\
        -\frac{\delta+i\gamma}{2\pi} f_2 -\frac{\delta-i\gamma}{2\pi} f_{-2}\|_{C^1}\leq \\
        \leq K(\alpha^2+\beta^2+\gamma^2+\delta^2)\leq K'\normC{\frac{\beta+i\alpha}{2\pi} f_1 +\frac{\beta-i\alpha2}{2\pi} f_{-1}+\frac{\delta+i\gamma}{2\pi} f_2 +\frac{\delta-i\gamma}{2\pi} f_{-2}}{1}^2.
    \end{gather*}
    (note that the linear combination of $f_{\pm 1},f_{\pm 2}$ we consider is real because we sum pairs of conjugates)
\end{proposition}
\begin{proof}
    We first bound the ($x$-) derivatives of the functions.
    Because mixed partial derivatives are equal, the $x$-derivatives of $f_{\pm 1},f_{\pm 2}$ give the first order approximation to the $x$-derivative of $V(A,B,C,D)$. 
    Using a (pointwise in $x$) Taylor expansion of the $x$-derivative of $V(A,B,C,D)$, we get an error of the form $O(\alpha^2+\beta^2+\gamma^2+\delta^2)$ on the derivative of the function in the left most side of the inequality we prove.
    These estimates are then made uniform because of analyticity of the function and the compactness of the interval.
    Integrating those inequalities from $0$ to $x$ gives the same desired inequality for the functions themselves.
    Finally, the right inequality follows since all norms on a finite dimensional space are equivalent, using Lemma \ref{lem:EquivNorm}.
\end{proof}
The following result ensures that if $\varepsilon$ is sufficiently small, the collection $(f_q)_q$ is a Riesz basis -- a collection of linearly independant vectors which spans a dense vector set in $L^2([0,1])$:

\begin{proposition}
\label{proposition:Hilbertbasis}
There is $\varepsilon_0>0$ such that for any $\varepsilon\in(0,\varepsilon_0)$, the collection $(f_q)_q$ is a Riesz basis for $L^2[0,1]$.
\end{proposition}

To prove Proposition \ref{proposition:Hilbertbasis}, we will need the following technical lemmae. The first is about the explicit values of $f_q$ for $q\in\{\pm1,\pm2\}$:

\begin{lemma}
\label{lemma:technical_fq_low} 
    If $q\in\{\pm1,\pm2\}$, when $\varepsilon\to 0$,
    \[
    \normC{f_q-E_q}{1}\to0
    \]
\end{lemma}

\begin{proof}
    Let us prove the result for $q=1$, the proof for the three other norms works in the same way. 
    Since $E_q=S_q+iC_q$, it is enough to show that
    \[
    \|\pi\partial_AV-C_1\|_{C^1}\to0
    \quad\text{and}\quad 
    \|\pi\partial_BV-S_1\|_{C^1}\to0.
    \]
    Let us show the first limit, since the second one is similar.
    For that we first consider the derivative
    A direct computation gives
    \[
    \pi\partial_AV'(x) = \frac{(1+F(x))\sin(2\pi x)+E(x)\cos(2\pi x)}{E(x)^2+(1+F(x))^2}
    \]
    where $E(x) = A\sin(2\pi x)+B\cos(2\pi x)+C\sin(4\pi x)+D\cos(4\pi x)$ and $F(x) = -A\cos(2\pi x)+B\sin(2\pi x)-C\cos(4\pi x)+D\sin(4\pi x)$. Note that the sup norm of $E$ and $F$ satisfy
    \[\|E\|_{\infty},\|F\|_{\infty}=\mathcal O(\varepsilon).\]
    Moreover,
    \[C_1'(x)=\sin(2\pi\theta_{\frac{1}{4}}(x))\theta_{\frac{1}{4}}'(x)=\sin(2\pi x)+L(x),\]
    where $\norm{L}_{\infty}=\mathcal{O}(\varepsilon)$.
    Therefore
    \[
    \pi\partial_AV'(x)-C'_1(x) = \frac{\sin(2\pi x)-\sin\left(2\pi \theta_{\frac{1}{4}}(x)\right)}{E(x)^2+(1+F(x))^2} +\mathcal O(\varepsilon)
    \]
    uniformly in $x$ as $\varepsilon\to 0$. It follows from item \ref{itm:theta14AtZero} in Proposition \ref{proposition:expansion_theta} that $\|\pi\partial_AV'-C'_1\|_{C^0}\to 0$ as $\varepsilon\to 0$. 
    Integrating this bound from $0$ to $x$ gives also
    \[\normC{\pi\partial_AV-C_1}{1}\to0.\]
\end{proof}

The second technical lemma is about estimates related to $f_q$ with $|q|\geq 3$.
In fact, we only need it for $|q|\geq 9$.
Introduce the functions $\tilde e_q$ defined for any $x\in\mathbb R$ by
\begin{equation}\label{eq:IntermediateBasisDefinition}
\tilde e_q(x) = e_q(x)U(x)^{t_q},    
\end{equation}

where $U$ is given by $U(x)=e^{\frac{i\pi}{2}u(x)}$, $u$ was introduced in Proposition \ref{proposition:expansion_theta}, and $t_q$ is the remainder of $q$ mod $4$.
Note that the latter proposition implies that $U\to 1$ uniformly in $x$ as $A,B\to 0$.

\begin{lemma}
\label{lemma:technical_fq_high}
    There is a constant $K=K(\varepsilon)>0$ which goes to $0$ as $\varepsilon\to0$ and such that for any $|q|\geq 9$
    \[
    \normC{f_q-\tilde e_q}{0}\leq \frac{K(\varepsilon)}{|q|}.
    \]
\end{lemma}

\begin{proof}
Fix $q\geq 9$ (the proof for negative $q$ is identical) and apply Proposition \ref{proposition:expansion_theta} to $f_q$ to obtain
    \[
    f_q(x) = e^{2i\pi q\theta_{r_q}(x)}\frac{\theta_{r_q}'(x)}{\theta_{\frac{1}{4}}'(x)}
     = e^{2i\pi q\left(\theta_{\frac{1}{4}}(x)+\left(r_q-\frac{1}{4}\right)u(x)+v(r_q,x)\right)}\frac{\theta_{r_q}'(x)}{\theta_{\frac{1}{4}}'(x)}
    \]
    which implies that
    \[
    f_q(x) = \tilde e_q(x) e^{2i\pi q v(r_q,x)}\frac{\theta_{r_q}'(x)}{\theta_{\frac{1}{4}}'(x)}.
    \] 
    Now by the choice of $r_q$ that we made (see \eqref{eq:RationalRotNumberCond}),
    \[
    e^{2i\pi q v(r_q,x)} = 1+\frac{R_q(x)}{q}
    \quad\text{and}\quad 
    \frac{\theta_{r_q}'(x)}{\theta_{\frac{1}{4}}'(x)} = 1+\frac{\tilde R_q(x)}{q}
    \]
    where $R_q(x)$ and $\tilde R_q(x)$ are bounded in $x$ and $q$ and converge to $0$ uniformly as $\varepsilon\to 0$.
    Therefore

    \[
    \left|f_q(x)-\tilde e_q(x)\right| = \left|e^{2\pi i q v(r_q,x)}\frac{\theta_{r_q}'(x)}{\theta_{\frac{1}{4}}'(x)}-1\right| = \left|\frac{R_q(x)}{q}+\frac{\tilde R_q(x)}{q}+\frac{R_q(x)\tilde R_q(x)}{q^2}\right|
    \] 
    and the result follows.
\end{proof}

\begin{proof}[Proof of Proposition \ref{proposition:Hilbertbasis}]
    Consider the linear map $T:L^2([0,1])\to L^2([0,1])$ defined by the property $T(e_q)=f_q$ for any $|q|=0,3,4,...$, and $T(E_q)=(f_q)$ for $|q|=1,2$.
    To prove the statement we prove that for sufficiently small $\varepsilon>0$, $T-I$ is a bounded operator of norm strictly smaller than $1$. 
    
    Given $\varphi\in L^{2}([0,1])$, we decompose it in our Fourier basis as $\varphi=\varphi_0e_0+\sum_{|q| = 1,2 }\varphi_qE_q+\sum_{|q|\geq 3}\varphi_q e_q$.
    Observe that the first two summands are orthogonal to the third.
    Then, the norm of $\varphi$ (with respect to the inner product \eqref{eq:NewInnerProduct}) is
    \[
    \|\varphi\|^2 = P(\varphi_0,\varphi_{\pm1},\varphi_{\pm2})+\sum_{|q|\geq 3}|\varphi_q|^2, 
    \]
    where $P$ is the square of a fixed (independent of the Suris parameters) norm on $\C^5$, which corresponds to the change of basis from $E_{\pm1,\pm2}$ to $e_{\pm1,\pm2}$.
    We can now write (using the fact that $T(e_0)=f_0=e_0$. In this proof, all norms denote the $L^2$ norm of \eqref{eq:NewInnerProduct})
\begin{equation}\label{eq:SeperrateHighLow}
    \|(T-I)(\varphi)\|
    \leq \sum_{|q| = \pm1,\pm2}|\varphi_q|\|f_q-E_q\|
    +\|\sum_{|q|\geq 3}\varphi_q(f_q-e_q)\|.
\end{equation}

    Let us handle these two terms separately.
    For the first term, we consider the following rough estimate
    \[
    \sum_{|q|=\pm1,\pm2}|\varphi_q|\|f_q-E_q\| \leq \left(\sum_{|q|=\pm1,\pm2} |\varphi_q|^2\sum_{|q|=\pm1,\pm2}\norm{f_q-E_q}^2\right)^{\frac{1}{2}}.
    \]
    We can bound the $L^2$ norm in the sum with the $C^1$ norm, and hence by Lemma \ref{lemma:technical_fq_low}, 
    \[K_1^2:=\sum_{|q|=\pm1,\pm2}\norm{E_q-f_q}^2\to0\]
    as $\varepsilon\to0$.
    Since all norms on $\C^5$ are equivalent, then we can replace (by enlarging $K_1$ is necessary) the Euclidean norm of the coefficients $\varphi_{\pm1,\pm2}$ by $P(\varphi_0,\varphi_{\pm1},\varphi_{\pm2})$.
    In total we get the estimate
    \begin{equation}\label{eq:LowOrderNormBound}
        \sum_{|q|=\pm1,\pm2}|\varphi_q|\norm{f_q-E_q}\leq K_1(\varepsilon)\left(P(\varphi_0,\varphi_{\pm1},\varphi_{\pm2})\right)^{\frac{1}{2}}\leq K_1(\varepsilon)\norm{\varphi}.
    \end{equation}
    
    For the second term of \eqref{eq:SeperrateHighLow}, denoted by $S$, we first split it again:
    \[S=\norm{\sum_{|q|\geq 3}\varphi_q(f_q-e_q)}\leq S_1+S_2,\]
    where 
    \[S_1=\norm{\sum_{3\leq|q|\leq 8}\varphi_q(f_q-e_q)},\,S_2=\norm{\sum_{|q|\geq 9}\varphi_q(f_q-e_q)}.\]
    As $\varepsilon\to0$, we have that both $e_q,f_q\to e^{2\pi iq x}$ (uniformly in $x$).
    Hence, for $S_1$, since we only have finitely many summands, we can find a constant $K_2(\varepsilon)\to0$ (as $\varepsilon\to 0)$ for which
    \[S_1\leq K_2(\varepsilon)\left(\sum_{3\leq |q|\leq 8} |\varphi_q|^2\right)^{\frac{1}{2}}\leq K_2(\varepsilon)\norm{\varphi}.\]
    To estimate $S_2$, we split it further by introducing $\tilde e_q$:
    \[
    S_2 = \|\sum_{|q|\geq 9}\varphi_q(f_q-e_q)\| \leq S_{21}+S_{22},
    \]
    where 
    \[
    S_{21} = \|\sum_{|q|\geq 9}\varphi_q(f_q-\tilde e_q)\|
    ,\,
    S_{22} = \|\sum_{|q|\geq 9}\varphi_q(\tilde e_q-e_q)\|. 
    \]
    To bound $S_{21}$, we use Cauchy-Schwarz inequality together with Lemma \ref{lemma:technical_fq_high}:
    \[
    S_{21}\leq \left(\sum_{|q|\geq 9}|\varphi_q|^2\sum_{|q|\geq 9}\norm{f_q-\tilde{e}_q}^2\right)^{\frac{1}{2}}\leq K_3(\varepsilon)\norm{\varphi}
    \]
    where $K_3(\varepsilon)=K(\varepsilon)\sqrt{2\sum\limits_{q=9}^\infty \frac{1}{q^2}}\to 0$ as $\varepsilon\to 0$ (and $K(\varepsilon)$ is the constant from Lemma \ref{lemma:technical_fq_high}).
    Finally, to estimate $S_{22}$, consider the expression of each $\tilde e_q$:
    \[
    S_{22} = \|\sum_{|q|\geq 9}\varphi_q(\tilde e_q-e_q)\|
    = \|\sum_{|q|\geq 9}\varphi_qe_q(U^{t_q}-1)\|\leq \sum_{\sigma=0}^3\|\sum_{t_q=\sigma}\varphi_qe_q(U^{\sigma}-1)\|
    \]
    where the last inequality has been obtained by gathering the terms for which $t_q=\sigma$ for the different values of $\sigma\in\{0,1,2,3\}$. 
    Note that for each $\sigma\in\set{0,1,2,3}$, the vector
    \[\varphi_\sigma:=\sum_{t_q=\sigma, |q|\geq 9}\varphi_qe_q\]
    is just the orthogonal projection of $\varphi$ on some subspace, so $\norm{\varphi_\sigma}\leq\norm{\varphi}$.
    Therefore
    \[
    S_{22} \leq \sum_{\sigma=0}^3\|(U^{\sigma}-1)\varphi_\sigma\| \leq \|\varphi\| K_4(\varepsilon)
    \]
    where $K_4(\varepsilon) = \sum_{\sigma=0}^3\|(U^{\sigma}-1)\|_{C^0}\to 0$ as $\varepsilon\to 0$ -- as a consequence of Proposition \ref{proposition:expansion_theta}.
Coming back to \eqref{eq:SeperrateHighLow}, we see that 
\[\norm{(T-I)(\varphi)}\leq (K_1+K_2+K_3+K_4)\norm{\varphi},\]
so $\|T-I\|\leq K_1+K_2+K_3+K_4\to0$ as $\varepsilon\to 0$ which implies the result.
\end{proof}

To finish this section, we observe that the $L^2$ norm of a function $W$ (with respect to the inner product \eqref{eq:NewInnerProduct}) can be bounded in terms of the inner products $\langle W,f_q\rangle$.
This is not a special case of the Bessel inequality since $\set{f_q}$ is not orthonormal.
We could have also adapted the Proposition to not have the restriction of $W\perp\set{f_0,f_{\pm1},f_{\pm2}}$, but we have no need for such a statement.
\begin{proposition}\label{prop:LikeParseval}
    There exists $C>0$ such that for every $\varepsilon<\varepsilon_0$ ($\varepsilon_0$ being defined in Proposition \ref{proposition:Hilbertbasis}), and every function $W\perp\set{f_0,f_{\pm1},f_{\pm2}}$, we have
    \[\norm{W}^2\leq C\sum_{|q|\geq 3}|\langle W,f_q\rangle|^2,\]
    where the norm on the left-hand side is the $L^2$ norm with respect to \eqref{eq:NewInnerProduct}.
\end{proposition}
\begin{proof}
    We use the oprator $T$ defined in the proof of Proposition \ref{proposition:Hilbertbasis}.
    We showed there that for $\varepsilon<\varepsilon_0$ this is an invertible operator of $L^2$.
    Hence its conjugate is also invertible.
    Hence, using Parseval's identity for the orthonormal basis $\set{e_q}$,
    \begin{gather*}
        \norm{W}^2=\norm{(T^*)\inv T^*W}^2\leq C\norm{T^* W}^2=C\sum_{q\in\Z}|\langle T^*W,e_q\rangle|^2 = \\
        =C\sum_{q\in\Z}|\langle W,Te_q\rangle|^2=C\sum_{|q|\leq 2}|\langle W,Te_q\rangle|^2+C\sum_{|q|\geq 3}|\langle W,Te_q\rangle|^2.
    \end{gather*}
    For $|q|\geq 3$ we have $Te_q=f_q$, as needed.
    For $|q|\leq 2$, we use the fact that $\set{e_0,e_{\pm1},e_{\pm2}}$ and $\set{e_0,E_{\pm1},E_{\pm2}}$ span the same subspace.
    This means that $Te_q$ is a linear combination of $f_0,f_{\pm1},f_{\pm2}$, and by assumption on $W$, the corresponding inner product vanishes.
    This finishes the proof.
\end{proof}

\section{Action estimates}\label{sec:actionEstimate}

Consider a rationally  integrable standard map $F_V$ in the interval $[\frac{1}{6},\frac{1}{3}]$, where the potential $V$ is close to being a Suris potential $V_S$ (of small enough eccentricity $\varepsilon$).
We set $W:=V-V_S$.
The rational integrability assumption implies that we have invariant curves of rotation numbers $r_q$ for all $q\geq 3$ (see \eqref{eq:RationalRotNumberCond}).
Then, from the general theory of twist maps, we know that the action (see \eqref{eq:ActionDef}) of orbits on an invariant curve that consists of periodic orbits is constant.
Therefore, our goal is to compare this constant value between the Suris map $F_S$ and $F_V$ (see, e.g., \cite[Theorem 3]{avila2016integrable}).
Namely, consider any $x_0\in[0,1]$. Then we have two $\frac{p}{q}$ periodic oribts assosicated to it: $x_0,x_1,...,x_q=x_0+p$, which is the $\frac{p}{q}$ periodic orbit starting at $x_0$ for $F_S$, and $x'_0=x_0,x'_1,...,x'_q=x_0+p$, which is the $\frac{p}{q}$ periodic orbit for $F_V$ starting at $x_0$.
We need the following lemma, that measures the deviation between the two orbits:
\begin{lemma}\label{lem:OrbitDeviation}
    There exists $r>0$ integer and a constant $C(\varepsilon)$ that depends only on the accenetricity such that if $\normC{W}{r}$ is small enough, for all $k=0,...,q-1$,
    \[|x_k-x'_k|\leq C(\varepsilon)q^2 \normC{W}{1}.\]
\end{lemma}
We prove the lemma later. 
Right now, we will use it to estimate the action deviation.
More precisely, we prove the following:
\begin{lemma}\label{lem:ActionEstimate}
    For $W$ as above and a given $x_0\in[0,1]$, if $F_V$ has a minimal\footnote{The orbits we consider are minimal as they lie on invariant curves of $F_V$} periodic orbit of rotation number $\frac{p}{q}$ starting at $x_0$, $x'_0=x_0,x'_1,...,x'_q=x_0+p$, then the following estimate holds true:
    \begin{equation}\label{eq:ActionEstimate}
        |A-A_S-\sum_{k=0}^{q-1}W(x_k)|\leq C(\varepsilon)q^5 \normC{W}{1}^2.
    \end{equation}
   Here $x_1,...,x_q$  are the coordinates of the $\frac{p}{q}$ periodic orbit for $F_S$ that starts at $x_0$, $A=A_V$ denotes the action of the periodic orbit for $F_V$ and $A_S=A_{V_S}$ denotes the action of the corresponding orbit for $F_S$. 
   Here $C(\varepsilon)$ is some constant that depends only on $\varepsilon$.
\end{lemma}
\begin{proof}
    We start with the computation of $A$.
    \begin{gather*}
        A = \sum_{k=0}^{q-1}\frac{1}{2}(x'_{k+1}-x'_k)^2+V(x'_k)=\sum_{k=0}^{q-1}\frac{1}{2}[(x_{k+1}-x_k)+(x'_{k+1}-x_{k+1}+x_k-x'_k)]^2+\\
        +V_S(x'_k)+W(x'_k)=\sum_{k=0}^{q-1}[\frac{1}{2}(x_{k+1}-x_k)^2+V_S(x'_k)]+\sum_{k=0}^{q-1}W(x'_k)+\\
        +\sum_{k=0}^{q-1}[(x'_{k+1}-x_{k+1}+x_k-x'_k)(x_{k+1}-x_k)+\frac{1}{2}(x'_{k+1}-x_{k+1}+x_k-x'_k)^2].
    \end{gather*}
    Write $\Delta_k=x'_k-x_k$, and use Taylor expansion of order $2$ of $V_S$ around $x_k$, and of order $1$ of $W$ around $x_k$:
    \begin{equation}\label{eq:TaylorVW}
        \begin{cases}
            V_S(x'_k)=V_S(x_k)+\Delta_kV'_S(x_k)+\frac{1}{2}V''(\xi_k)\Delta_k^2,\\
            W(x'_k)=W(x_k)+\Delta_k W'(\eta_k),
        \end{cases}
    \end{equation}
    for some unknown $\xi_k,\eta_k$.
    Therefore we can rearrange the above formula:
    \begin{multline}\label{eq:AButNice}
        A=\sum_{k=0}^{q-1}[\frac{1}{2}(x_{k+1}-x_k)^2+V_S(x_k)+W(x_k)]\\
        +\sum_{k=0}^{q-1}[(\Delta_{k+1}-\Delta_k)(x_{k+1}-x_k)+V_S'(x_k)\Delta_k]+\\
        +\sum_{k=0}^{q-1}[\frac{1}{2}(\Delta_{k+1}-\Delta_k)^2+\frac{1}{2}V_S''(\xi_k)\Delta_k^2+W'(\eta_k)\Delta_k.]
    \end{multline}
    The first two summands in the top line of \eqref{eq:AButNice} sum to $A_S$, and the sum over $W(x_k)$ appears in \eqref{eq:ActionEstimate}.
    Hence, the left hand side of \eqref{eq:ActionEstimate} is equal to the bottom two rows of \eqref{eq:AButNice}.
    Next we claim that the middle row of \eqref{eq:AButNice} vanishes.
    Indeed,
    \begin{gather*}
        \sum_{k=0}^{q-1}(\Delta_{k+1}-\Delta_k)(x_{k+1}-x_k)=\sum_{k=0}^{q-1}\Delta_{k+1}(x_{k+1}-x_k)-\sum_{k=0}^{q-1}\Delta_k(x_{k+1}-x_k)=\\
        =\sum_{k=1}^{q}\Delta_k(x_k-x_{k-1})-\sum_{k=0}^{q-1}\Delta_k(x_{k+1}-x_k)=\Delta_q(x_q-x_{q-1})-\\
        -\Delta_0(x_1-x_0)-\sum_{k=1}^{q-1}\Delta_k(x_{k+1}-2x_k+x_{k-1}).
    \end{gather*}
    Note that $\Delta_q=\Delta_0=0$ since both periodic orbits have the same start and end point.
    Next, since $\set{x_k}$ is an orbit for the map $F_S$, we have the Frenkel-Kontorova equation:
    \[x_{k+1}-2x_k+x_{k-1}=V_S'(x_k)\]
    Hence the middle row of \eqref{eq:AButNice} sums to zero.
    To estimate the last row, we use Lemma \ref{lem:OrbitDeviation}, to conclude that $|\Delta_k|\leq C(\varepsilon)q^2 \normC{W}{1}$.
    This clearly implies that the bottom row of \eqref{eq:AButNice} is bounded by the right hand side of \eqref{eq:ActionEstimate}.
    
\end{proof}

Now we would like to prove Lemma \ref{lem:OrbitDeviation}.
The idea of the proof is to estimate the deviation of an orbit of $F_V$ 
in action-angle coordinates adapted to $F_S$. 
More precisely, we consider a change of coordinates 
$\psi(\theta,I)=(x(\theta,I),y(\theta,I))$ 
such that
\begin{equation}
    \label{eq:ActAngFS}
    \psi^{-1}\circ F_S\circ \psi (\theta,I)=(\theta+\omega(I),I).
\end{equation}
It follows that, where it is defined we can write
    \begin{equation}
    \label{eq:ActAngFV}
    F_V\circ\psi(\theta,I)=F_S\circ\psi(\theta,I)+W'(x(\theta,I))(1,1)=\psi(\theta+\omega(I),I)+W'(x(\theta,I))(1,1),
    \end{equation}
    and therefore
    \[\psi^{-1}\circ F_V\circ \psi(\theta,I)=(\theta+\omega(I)+R_1(\theta,I),I+R_2(\theta,I))\]
    with $\normC{R_1}{1},\normC{R_2}{1}\leq C\normC{W}{1}$ and the constant $C$ depends only on the Suris potential $V_S$.

    We can assume that $\psi$ is defined on an open set $\mathcal U$ containing $[I_0,I_1]\times \R$ where $I_0<I_1$ define two invariant curves $\gamma_0 = \im \psi(I_0,\cdot)$  and $\gamma_1 = \im \psi(I_1,\cdot)$ of $F_S$ of respective Diophantine rotation numbers $\omega_0<\omega_1$ with $\omega_0\in (0,1/6)$ and $\omega_1\in (1/3,1/2)$.

\begin{lemma}
     There exists $\delta>0$ and an integer $r>0$ such that if $\normC{V-V_S}{r}\leq \delta$, any minimizing orbit of $F_V$ is included in $\mathcal U$.
\end{lemma}

\begin{proof}
    Let us consider the following proposition which follows from general KAM theory on area-preserving twist maps \cite{Herman,Poeschel2}:

    \begin{proposition}[KAM theorem]
    \label{proposition:KAM_theorem}
        Let $\ell>0$ be an integer. There exists $\delta>0$ and an integer $r>0$ such that if
        \[
        \|V-V_S\|_{ C^r} \leq \delta
        \]
        then $F_V$ has also two invariant curves $\gamma_0'$ and $\gamma_1'$ of rotation numbers $\omega_0$ and $\omega_1$ which are $ C^{\ell}$ smooth and such that 
        \[
        \|\gamma_i'-\gamma_i\|_{ C^{\ell}} \qquad\qquad i=0,1
        \]
        can be taken arbitrary small when $\|V-V_S\|_{ C^r}\to 0$.
    \end{proposition}

    We can therefore fix an arbitrary $\ell>0$ -- since $\gamma_0$ and $\gamma_1$ are analytic, and consider $r>0$ and $\delta>0$ as in Proposition \ref{proposition:KAM_theorem} and such that the curves $\gamma_0',\gamma_1'$ are included in $\mathcal U$ for $\|V-V_S\|_{ C^r} \leq \delta$. 

    Now the region delimited by the curve $\gamma_0'$ and $\gamma_1'$ is invariant by $F_V$ and included in $\mathcal U$. Hence any minimal set of $F_V$ of rotation number $\omega\in(\omega_0,\omega_1)$ is contained in $\mathcal U$. 
\end{proof}

Until the rest of the proof, we assume that $\normC{V-V_S}{r}\leq \delta$ so that the minimal orbits of $F_V$ of rotation numbers in $(\omega_0,\omega_1)$ lie in $\mathcal{U}$.

    \begin{lemma}
    \label{lemma:EstWithoutInit}
        There exists a constant $C=C(V_S)>0$ such that for all $k=0,...,q-1$,
        \[
        |x'_k-x_k| \leq C(k|y_0'-y_0|+k^2\normC{W}{1}).
        \]
    \end{lemma}

    \begin{proof}
        Denote by $(\theta_k,I_k)=\psi^{-1}(x_{k+1},y_{k+1})$ (where $I_k=I_0$ for all $k$) and $(\theta'_k,I'_k)=\psi^{-1}(x'_{k+1},y'_{k+1})$, for all $k=0,\ldots,q-1$.
    We also denote by $L$ the Lipschitz constant of the function $\omega$.
    From the expressions of $F_S$ and $F_V$ in $(I,\theta)$ coordinates given in Equations \eqref{eq:ActAngFS} and \eqref{eq:ActAngFV}, we have the following estimates:
    \[      |I_{k+1}-I'_{k+1}|=|I_k-(I_k'+R_2)|\leq |I_k-I'_k|+C\normC{W}{1}. \]
    This implies that
    \begin{equation}\label{eq:IkEstimate}
         |I_k-I'_k|\leq |I_0-I'_0|+k\normC{W}{1}.
    \end{equation}
    
Similarly,
    \begin{multline*}
|\theta_{k+1}'-\theta_{k+1}| =|(\theta'_k+\omega(I'_k)+R_1)-(\theta_k+\omega(I_k))|\leq \\
\leq|\theta_k'-\theta_k|+L|I_k'-I_k| + C\|W\|_{C^1}
\leq |\theta_k'-\theta_k| +L|I_0'-I_0|+ (k+1)C'\|W\|_{C^1},
\end{multline*}
which implies that 
\begin{equation}\label{eq:ThetaKDeviation}
    |\theta'_k-\theta_k|\leq |\theta'_0-\theta_0|+Lk|I'_0-I_0|+C'k^2\normC{W}{1}.
\end{equation}
Finally, if $M$ denotes the larger of the $C^1$ norm of $\psi$ and $\psi^{-1}$ (which depends only on $V_S$), we get, using \eqref{eq:IkEstimate},\eqref{eq:ThetaKDeviation} that
\begin{multline}\label{eq:DeriveXkEstimate}
 |x'_k-x_k|\leq M(|\theta_{k-1}-\theta'_{k-1}|+|I_{k-1}-I'_{k-1}|)\\
 \leq M(|\theta_0-\theta'_0|+ Lk|I_0-I'_0|+C'k^2\normC{W}{1})\\
 \leq MLk(|\theta_0-\theta'_0|+|I_0-I'_0|)+C'k^2\normC{W}{1}\leq M^2Lk(|x_0-x'_0|+|y_0'-y_0|)+C'k^2\normC{W}{1}.
\end{multline}
The result follows by setting $C = MLk$ and observing that $x_0=x_0'$.
    \end{proof}

It remains to show the following lemma: 

\begin{lemma}
    \label{lemma:EstInitCond}
    There exists $C=C(V_S)>0$ for which
        \begin{equation}\label{eq:EstimateAtPoint}
            |y_0-y'_0|\leq Cq\normC{W}{1}.
        \end{equation}
\end{lemma}

    The idea is to estimate how much an orbit for $F_V$ fails to be an orbit for $F_S$.
    For that we define 
    \[\begin{cases}
        y_k^+=x'_{k+1}-x'_k-V_S'(x'_k),\\
        y_k^-=x'_k-x'_{k-1},
    \end{cases}\]
    and we observe that it holds that
    \begin{equation}\label{eq:AlmostOrbit}
        F_S(x'_k,y_k^+)=(x'_{k+1},y_{k+1}^-).
    \end{equation}    
    By Birkhoff's Theorem, the invariant curve of rotation number $\frac{p}{q}$ is orbits is a graph of a function $y=y_{\frac{p}{q}}(x)$.
    \begin{claim}\label{claim:invCurveBetween}
        For some $k$, it holds that 
        \[y_k^-\leq y_{\frac{p}{q}}(x'_k)\leq y_{k}^+.\]
    \end{claim}
    \begin{proof}
        Arguing by contradiction, we assume that one of the inequalities holds in the opposite direction for all $k$, for example that $y_k^->y_{\frac{p}{q}}(x_k')$.
        By \eqref{eq:AlmostOrbit} and the fact that $F_S$ is orientation preserving, it is impossible for an orbit to cross the two sides of an invariant curve.
        Therefore, if $y_k^->y_{\frac{p}{q}}(x'_k)$, then we also have $y_{k-1}^+ > y_{\frac{p}{q}}(x'_{k-1})$.
        Hence we get that $y_k^{\pm}\geq y_{\frac{p}{q}}(x'_k)$.
        Denote by $\psi:S^1\to S^1$ the restriction of $F_S$ to the invariant curve $y=y_{\frac{p}{q}}(x)$, and $\tilde{\psi}$ its lift to $\R$, which is strictly monotone. It satisfies $F(x,y_{\frac{p}{q}}(x))=(\tilde\psi(x),y_{\frac{p}{q}}(\tilde\psi(x)))$ for any $x\in\R$.
        Then by the twist condition and Equation \eqref{eq:AlmostOrbit}, we get that
        $x'_k<\tilde{\psi}(x'_k)<x'_{k+1}$.
        Iterating this argument $q$ times, we get a contradiction to the fact that $\tilde{\psi}^q(x)=x+p$ for all $x\in\R$. 
    \end{proof}
    
We can now prove Lemma \ref{lemma:EstInitCond}:
\begin{proof}[Proof of Lemma \ref{lemma:EstInitCond}]
    For all $k=0,...,q$, we have
    \[|y_k^+-y_k^-|=|x'_{k+1}-2x'_k+x'_{k-1}-V_S'(x'_k)|=|W'(x'_k)|.\]
    Here we used the fact that $\set{x'_k}$ is a solution to the Frenkel-Kontorova equation with potential $V_S+W$.
    From this it follows that for all $k$,
    \[|y_k^+-y_k^-|\leq \normC{W}{1}.\]
    Moreover, if we choose $k_0$ as per Claim \ref{claim:invCurveBetween}, then we also have
    \[y_{k_0}^+-y_{\frac{p}{q}}(x_{k_0}')\leq y_{k_0}^+-y_{k_0}^-\leq \normC{W}{1}.\]
    Now we consider the value of the integral $I$ (see \eqref{equation:first_integral_I_2}) on these points:
    \[I_k^{\pm}=I(x'_k,y_k^{\pm}),\]
    and we also consider $I_0=I(x,y_{\frac{p}{q}}(x))$ 
    (this does not depend on $x$ since the invariant curve $y=y_{\frac{p}{q}}(x)$ is a level set of $I$).
    Then
    \begin{equation}\label{eq:IntegralIntegralIneq}
        |I_{k_0}^--I_0|\leq\int\limits_{y_{k_0}^+}^{y_{\frac{p}{q}}(x'_{k_0})}|\frac{\partial I}{\partial y}(x'_{k_0},\eta)|d\eta.
    \end{equation}
    It holds that
    \[\frac{\partial I}{\partial y}=2\pi[\sin(2\pi y)+A\sin(2\pi(x-y))+B\cos(2\pi(x-y))+C\sin(2\pi(2x-y))+D\cos(2\pi(2x-y))].\]
    Hence $|\frac{\partial I}{\partial y}|\leq 2\pi(1+4\varepsilon)$.
    Hence
\begin{equation}\label{eq:integralDeviation}
        |I_{k_0}^--I_0|\leq 2\pi(1+4\varepsilon)\normC{W}{1}.
\end{equation}
     Moreover, the same right hand side is also an upper bound for $|I_k^+-I_k^-|$ for all $k$.
    Also, from \eqref{eq:AlmostOrbit}, and the fact that $I$ is conserved quantity, $I_k^+=I_{k+1}^-$ for all $k$.
    From this we get  that for all $k$,
    \[|I_{k+1}^--I_0|\leq |I_{k+1}^+-I_{k+1}^-|+|I_k^--I_0|.\]
    Then we can chain this inequality with \eqref{eq:integralDeviation}, to get that for all $k$,
    \begin{equation}\label{eq:IntegralDeviationRecursive}
        |I_k^--I_0|\leq 2\pi q(1+4\varepsilon)\normC{W}{1}.
    \end{equation}
    On the other hand, since $y_k'=y_k^-$, applying the Mean Value Theorem together with  \eqref{eq:IntegralDeviationRecursive} for $k=0$ leads to 
    \[
    |y_0'-y_0|\leq\frac{1}{\kappa}|I_0^--I_0|\leq \frac{Cq}{\kappa}\|W\|_{ C^1}
    \]
    where 
    \[
    \kappa = \inf_{(x,y)\in\mathcal U'} |\partial_y I(x,y)|
    \]
    and $\mathcal U'$ is the region bounded by $\gamma_0$ and $\gamma_1$. Note that $\kappa$ depends only on $V_S$, $\omega_0$ and $\omega_1$. Let us estimate it as $\varepsilon\to 0$. We computed that 
    \[
    \partial_y I(x,y) = 2\pi\sin(2\pi y)+\mathcal O(\varepsilon).
    \]
    For $\varepsilon$ sufficiently small, the graphs $\gamma_0$ and $\gamma_1$ are close to being the constants $\omega_0$ and $\omega_1$, respectively.
    Hence
    \[
    \mathcal U'\subset [\omega_0+\mathcal O(\varepsilon), \omega_1+\mathcal O(\varepsilon)]
    \]
    from which we deduce
    \[
    \kappa \geq 2\pi\min_{i=0,1} \sin(2\pi\omega_i)+\mathcal O(\varepsilon).
    \]
    If we define $\kappa_0 = \pi\min_{i=0,1} \sin(2\pi\omega_i)>0$, then $\kappa>\kappa_0$ whenever $\varepsilon>0$ is sufficiently small.
    
    It remains to show that the orbit $(x_k',y_k')_{k}$ lies in $\mathcal U'$ if $W$ is sufficiently small. To do so, it is enough to show that this holds for one $k$. This is a consequence of Claim \ref{claim:invCurveBetween}: since $y_{k_0}^- = y_{k_0}'$, the claim implies that 
    \[
    \text{dist}((x_{k_0}',y_{k_0}'),\text{graph}(y_{\frac{p}{q}}))\leq \|W\|_{ C^0}
    \]
    and hence the result. 
    As a consequence,
    \[
    |y_0'-y_0|\leq\frac{1}{\kappa_0}|I_0^--I_0|\leq \frac{Cq}{\kappa_0}\|W\|_{ C^1}.
    \]
\end{proof}

We can finally prove Lemma \ref{lem:OrbitDeviation}:

\begin{proof}[Proof of Lemma \ref{lem:OrbitDeviation}]
    This is the consequence of Lemma \ref{lemma:EstWithoutInit} combined with Lemma \ref{lemma:EstInitCond}.
\end{proof}

Now we turn to the main result of this section.
\begin{theorem}\label{thm:ActionEstimate}
    Suppose that $F_{V_S+W}$ has an invariant curve of rotation number $r_{|q|} = \frac{p_{|q|}}{|q|}$ (see \eqref{eq:RationalRotNumberCond}), for $|q|\geq 3$. 
    Then, if the eccentricity of $V_S$ is small enough (so that Proposition \ref{proposition:Hilbertbasis} holds), we have:
    \begin{equation}
        |\langle W,f_q\rangle|\leq C(\varepsilon)q^4\normC{W}{1}^2.
    \end{equation}
\end{theorem}
\begin{proof}
    We assume that $q\geq 3$, and the negative case is identical.
    Fix an arbitrary $\theta_0\in[0,1]$, and let $x_0 = \theta_{r_q}(\theta_0)$.
    Then by Lemma \ref{lem:ActionEstimate}, we have
    \[|A-A_S-\sum_{k=0}^{q-1}W(x_{r_q}(\theta_0+kr_q))|\leq C(\varepsilon)q^5\normC{W}{1}^2,\]
    where here $A$ and $A_S$ denote the actions of of those $q$ periodic orbits starting at $x_0$ for $F_{V_S+W}$ and $F_{V_S}$, respectively.
    Since the invariant curve exists, the action does not depend on the choice of $\theta_0$.
    Hence,
    \[\int\limits_0^1 \sum_{k=0}^{q-1}W(x_{r_q}(\theta_0+kr_q))e^{-2\pi iq\theta_0}d\theta_0=\int_0^1[A_S+\sum_{k=0}^{q-1}W(x_{r_q}(\theta_0+kr_q)-A]e^{-2\pi iq\theta_0}d\theta_0.\]
    Hence by Lemma \ref{lem:ActionEstimate}, we have
    \[|\int\limits_0^1\sum_{k=0}^{q-1}W(x_{r_q}(\theta_0+kr_q))e^{-2\pi iq\theta_0}d\theta_0|\leq C(\varepsilon)q^5\normC{W}{1}^2.\]
    By using the change of varaibles $\theta\mapsto \theta+r_q$, we see that all summands in the left hand side have the same integral, so  we get
    \[|\int\limits_0^1W(x_{r_q}(\theta_0))e^{-2\pi iq\theta_0}d\theta_0|\leq C(\varepsilon)q^4\normC{W}{1}^2.\]
    In this integral we now change variables, $x=x_{r_q}(\theta_0)\leftrightarrow \theta_0=\theta_{r_q}(x)$:
    \[|\int\limits_0^1 W(x)e^{-2\pi i q \theta_{r_q}(x)}\theta'_{r_q}(x)dx|\leq C(\varepsilon)q^4\normC{W}{1}^2.\]
    But the left hand side is exactly $|\langle W,f_q\rangle|$ with respect to the inner product defined in \eqref{eq:NewInnerProduct}. 
\end{proof}
We also show another bound on the Fourier coefficients, which follows from the regularity of $W$.
This estimate is more useful as $q\to\infty$.
\begin{proposition}\label{prop:FourierCoefTail}
    There exists a constant $C(\varepsilon)$ such that for all $|q| \geq 9$
    \[|\langle W,f_q\rangle| \leq \frac{C(\varepsilon)\normC{W}{1}}{q}.\]
\end{proposition}
\begin{proof}
    We again show this for $q>0$. 
    First we compute $\langle W,\tilde{e}_q\rangle$.

    \[\langle W,\tilde{e}_q\rangle = \int\limits_0^1 W(x)e^{-2\pi i q\theta_{\frac{1}{4}}(x)}\overline{U(x)}^{s_q}\theta_{\frac{1}{4}}'(x)dx=\int\limits_0^1 (W\cdot \overline{U}^{s_q})\circ x_{\frac{1}{4}}(\theta)e^{-2\pi i q \theta}d\theta.\]
    Hence the result is the $q$-th Fourier coefficient of $(W\cdot \overline{U}^{s_q})\circ x_{\frac{1}{4}}$ with respect to the standard Fourier basis of $L^2$.
    By the regularity of this function, we know that there is some constant $c$ such that
    \[|\langle W,\tilde{e}_q\rangle|\leq \frac{c\normC{(W\cdot\overline{U}^{s_q})\circ x_{\frac{1}{4}}}{1}}{q}.\]
    Since $x_{\frac{1}{4}}$ is a diffeomorphism, by replacing $c$ with a larger constant, we can replace $\normC{(W\cdot\overline{U}^{s_q})\circ x_{\frac{1}{4}}}{1}$ with $\normC{W\cdot \overline{U}^{s_q}}{1}$.
    From the definition of $U$ (see \eqref{eq:IntermediateBasisDefinition}), it holds that $|U|=1$. 
    Hence $\max|W\cdot \overline{U}^{s_q}|=\max|W|$, and 
    \[\max|(W\cdot \overline{U}^{s_q})'|=\max |W'\cdot \overline{U}^{s_q}+W\cdot (\overline{U}^{s_q})'|\leq 2c\normC{W}{1},\]
    for some constant $c$.
    So as a result we gather that 
    \begin{equation}\label{eq:eqTildeBound}
        |\langle W,\tilde{e}_q\rangle|\leq \frac{c\normC{W}{1}}{q}.
    \end{equation}
    Now, using Lemma \ref{lemma:technical_fq_high} and \eqref{eq:eqTildeBound} we can finish:
    \begin{gather*}
        |\langle W,f_q\rangle| \leq |\langle W,\tilde{e}_q-f_q\rangle|+|\langle W,\tilde{e}_q\rangle|\leq \int\limits_0^1 |W||\tilde{e}_q-f_q||\theta'_{\frac{1}{4}}(x)|dx + \frac{c\normC{W}{1}}{q}\leq \\
        \leq \frac{C(\varepsilon)\normC{W}{1}}{q}+\frac{c\normC{W}{1}}{q}.
    \end{gather*}
    This is exactly as desired. 
    
\end{proof}
To conclude this section, we show that while the basis $\set{f_q}$ is not necessarily orthonormal, we do have orthogonality of harmonics $f_q$ with $|q|\geq 3$ with the Harmonics $f_{\pm 1,\pm 2}$.
\begin{proposition}\label{prop:HighLowOrthognality}
    For all $|q|\geq 2$, and $j=\pm 1,\pm 2$, $\langle f_q,f_j\rangle =0$. 
\end{proposition}
\begin{proof}
    We prove for example for $j=\pm 1$.
    Let $\delta>0$ be an arbitrary positive number.
    Consider the  (real valued) $1$-periodic function  $W=\frac{\delta i}{2\pi}(f_1-f_{-1})$.
    Then by Proposition \ref{prop:fpm1pm2AreDerivatives}, there is a Suris potential $\tilde{V}_S$ for which 
    \[\normC{\tilde{V}_S-V_S-W}{1}\leq K\delta^2.\]
    We view $\tilde{V}_S$ as a perturbation of $V_S$. 
    This perturbation is integrable since it is a Suris perturbation.
    Hence by Theorem \ref{thm:ActionEstimate}, we have, for $|q|\geq 3$,
    \[|\langle \tilde{V}_S-V_S,f_q\rangle|\leq C(\varepsilon)q^4 \normC{\tilde{V}_S-V_S}{1}^2.\]
    By the triangle inequality,
    \[\normC{V_S-\tilde{V}_S}{1}\leq \normC{\tilde{V}_S-V_S-W}{1}+\normC{W}{1}\leq K\delta,\]
    and
    \[|\langle \tilde{V}_S-V_S,f_q\rangle|\geq |\langle W,f_q\rangle|-|\langle \tilde{V}_S-V_S-W,f_q\rangle|.\]
    Together we get 
    \[|\langle W,f_q\rangle|\leq C(\varepsilon)q^4\delta^2+|\langle \tilde{V}_S-V_S-W,f_q\rangle|.\]
    The inner product is bounded by $\normC{\tilde{V}_S-V_S-W}{1}$, using the Cauchy-Schwarz inequality.
    The conclusion is that 
    \[|\langle W,f_q\rangle|\leq Kq^4\delta^2.\]
    This holds for all $|q|\geq 3$ and all $\delta>0$ small enough. 
   Then given $|q|\geq 3$, we can choose all $\delta>0$ small enough so that $|q|\leq \delta^{-\frac{1}{4}}\leftrightarrow\delta<|q|^{-4}$, and then we get 
   \[|\langle \frac{i}{2\pi}(f_1-f_{-1}),f_q\rangle  \leq K\delta,\]
   and thus, this inner product is zero.
   Then we can repeat the same argument with $W=\frac{\delta}{2\pi}(f_1+f_{-1})$, and together we get that 
   \[\langle f_1,f_q\rangle=\langle f_{-1},f_q\rangle =0.\]
\end{proof}
\section{Proof of Theorem \ref{thm:LocalRigidity}}\label{sec:ProofOfMainThm}
We are now in position to prove Theorem \ref{thm:LocalRigidity}.
The proof follows the same idea that was used for billiards in \cite{avila2016integrable}.
Suppose we are given a rationally integrable map of the form $F = F_{V_S+W}$.
The crucial step is to find another Suris potential $\tilde{V}_S$ which is much closer to $V_S+W$ than $V_S$.
This is done by considering the ``Fourier coefficients" of $W$ with respect to our basis up to order $2$.
\begin{proposition}\label{prop:ProjectPotential}
    If $V_S+W$ is the potential of an integrable twist map, with $V_S$ having sufficiently low eccentricity, and $\normC{W}{r}$ being small enough, for $r\geq 23$, then there exists a Suris potential $\tilde{V}_S$ for which $\tilde{W}=V_S+W-\tilde{V}_S$ satisfies
    \begin{equation}\label{eq:ProjectionEstimate}
        \normC{\tilde{W}}{1}\leq C(\normC{W}{23})\normC{W}{1}^{\frac{231}{230}},
    \end{equation}
    where $C(\cdot)$ is some monotone function.
\end{proposition}
\begin{proof}
    Since the eccentricity of $V_S$ is small enough, then by Proposition \ref{proposition:Hilbertbasis}, the collection $(f_q)_q$ of Definition \ref{def:FourierBasis} is in fact a basis.
    Denote by $A,B,C,D$ the parameters of $V_S$.
    Find $\alpha,\beta,\gamma,\delta,W_0\in\R$ for which
    \[P=W_0+\frac{1}{2\pi}((\beta+i\alpha)f_1+(\beta-i\alpha)f_{-1}+(\delta+i\gamma)f_2+(\delta-i\gamma)f_{-2})\]
    is the projection of $W$ onto the space spanned by $\set{f_0,f_{\pm1},f_{\pm2}}$.
    Since the dynamics is unaffected by shifts by constant of the potential, we may assume that $W_0=0$.
    Then if we consider the Suris potential $\tilde{V}_S$, with parameters $A+\alpha,B+\beta,C+\gamma,D+\delta$, then by Proposition \ref{prop:fpm1pm2AreDerivatives}, we get
    \[\normC{V_S-(\tilde{V}_S-P)}{1}\leq K(\alpha^2+\beta^2+\gamma^2+\delta^2)\leq K'\normC{P}{1}^2.\]
    Denote by $\norm{\cdot}_S$ the norm with respect to the inner product \eqref{eq:NewInnerProduct}, defined with the help of the potential $V_S$.
    By Proposition \ref{prop:HighLowOrthognality}, we have
    \[\norm{W}_S^2=\norm{P}_S^2+\norm{W-P}_S^2.\]
     Hence in total, for $\tilde{W}:=V_S+W-\tilde{V}_S$,
\begin{equation}\label{eq:WTildeEstimate}
    \normC{\tilde{W}}{1}=\normC{(V_S+P-\tilde{V}_S)+(W-P)}{1}\leq K'\normC{P}{1}^2+\normC{W-P}{1}.    
\end{equation}
Using Lemma \ref{lem:EquivNorm},  $\normC{P}{1}$ is comparable with $\normL{P}{2}$ up to a multiplicative constant, and thus by changing $K'$ to another constant, this is bounded by $\normL{W}{2}^2$.

    According to Proposition \ref{prop:LikeParseval}, we have 
    \begin{equation}\label{eq:ApplicationOfBessel}
        \norm{W-P}_S^2\leq C\sum_{|q|\geq 3} |\langle W-P,f_q\rangle|^2=C\sum_{|q|\geq 3} |\langle W,f_q\rangle|^2.
    \end{equation}
    Here we used again the orthogonality of Proposition \ref{prop:HighLowOrthognality},
\[  \langle W-P,f_q\rangle=\langle W,f_q\rangle.\]
    For the right hand side of \eqref{eq:ApplicationOfBessel}, we set $q_0=\max\left(\left\lfloor\normC{W}{1}^{-\frac{1}{5}}\right\rfloor,8\right)$, and we split the sum:
    \[\sum_{|q|\geq 3} |\langle W,f_q\rangle|^2=\sum_{|q|=3}^{q_0}|\langle W,f_q\rangle |^2+\sum_{|q|=q_0+1}^\infty |\langle W,f_q\rangle |^2.\]
    For the finite sum we use the fact that $\normC{W}{r}$ is small enough, so we can use Theorem \ref{thm:ActionEstimate}:
    \begin{equation}\label{eq:FiniteSumEstimate}
        \sum_{|q|=3}^{q_0}|\langle W,f_q\rangle|^2\leq C\normC{W}{1}^4q_0^9.
    \end{equation}
    For the infinite sum, we use Proposition \ref{prop:FourierCoefTail}:
    \begin{equation}\label{eq:SeriesTailEstimate}
        \sum_{|q|=q_0+1}^\infty|\langle W,f_q\rangle |^2\leq \frac{C\normC{W}{1}^2}{q_0}.
    \end{equation}
    Our choice of $q_0$ guarantees that both \eqref{eq:FiniteSumEstimate} and \eqref{eq:SeriesTailEstimate} are bounded by $C\normC{W}{1}^{\frac{11}{5}}$, and hence 
    \begin{equation}\label{eq:SNormWminusP}
        \norm{W-P}_S\leq C\normC{W}{1}^{\frac{11}{10}}.
    \end{equation}
    We transform this estimate into the estimate on the $C^1$ norm of $W-P$ using the Sobolev interpolation inequality (see, e.g., \cite[Theorem 7.28]{gilbarg1977elliptic}).
    Also note that because $\theta'_{\frac{1}{4}}(x)$ is positive, bounded, and bounded away from zero, the norm $\norm{\cdot}_S$ is equivalent to the usual $L^2$ norm.
    Hence we can use these norms interchangebly in the Sobolev interpolation inequality.
    Hence, for $j=1,2$, $m\in\N$, and $\delta>0$ we have
    \[\normL{(W-P)^{(j)}}{2}\leq C(\delta\normC{W-P}{m}+\delta^{-j\slash(m-j)}\norm{W-P}_S).\]
    The stricter inequality is for $j=2$. 
    In this case we can put $m=23$ and $\delta=\normC{W}{1}^{\frac{231}{230}}$, and we get
    
    \[  \normL{(W-P)'}{2},\normL{(W-P)''}{2}\leq   C\normC{W}{1}^{\frac{231}{230}}(1+\normC{W-P}{23}).\]

    From which it follows also that 
    \begin{equation}\label{eq:L2C1FunnyPower}
    \normC{W-P}{1}\leq C\normC{W}{1}^{\frac{231}{230}}(1+\normC{W-P}{23}). 
    \end{equation}
    Now we come back to \eqref{eq:WTildeEstimate}, and the remark below it, to get
    \[\normC{\tilde{W}}{1}\leq C(\normL{W}{2}^2+\normC{W}{1}^{\frac{231}{230}}(1+\normC{W-P}{23})).\]
    The $L^2$ norm of $W$ is bounded by its $C^{23}$ norm. 
   Using Lemma \ref{lem:EquivNorm}, the $C^{23}$ norm is equivalent to the $S$-norm, so
    \[\normC{W-P}{23}\leq \normC{W}{23}+\normC{P}{23}\leq \normC{W}{23}+M\norm{P}_S.\]
    But since $P$ is the projection of $W$, it is shorter, so it is bounded by $\norm{W}_S$, which is in turn bounded by the $C^{23}$ norm.
    Combining this all, we give an inequality of the form
    \[\normC{W-P}{23}\leq M\normC{W}{23},\]
    for some constant $M$ that does not depend on eccentricity, as long as it is small enough.
    Thus, all together we get the required result,
    \[\normC{\tilde{W}}{1}\leq C(\normC{W}{23})\normC{W}{1}^{\frac{231}{230}},\]
    where $C(\cdot)$ is a monotone function.
\end{proof}

Finally, let us consider a Suris potential $V_S$ with eccentricity small enough so that Proposition \ref{proposition:Hilbertbasis} holds.
Suppose that $W$ is a $C^{r}$ function for which the $C^r$ norm is smaller than some $0<\delta<1$ that will soon be specified, and for which the $C^{23}$ norm is smaller than $1$. 
We assume that $F_{V_S+W}$ is rationally integrable. 
As long as $\delta$ is small enough, if $\tilde{V}_S$ is a Suris potential and $\normC{V_S-\tilde{V}_S}{1}<\delta$  then the eccentricity of $\tilde{V}_S$ is also small enough for Proposition \ref{proposition:Hilbertbasis} to hold.
Moreover, we may shrink $\delta$ so that this ball is contained in the ball of radius $1$ around $V_S$ in the $C^{23}$ norm.
The map that assigns to a tuple $(A,B,C,D)$ with $A^2+B^2+C^2+D^2\leq\frac{1}{2}$ its corresponding Suris potential is a continuous injective map from a compact space (we can always assume this assumption on the eccentricity by restricting $\varepsilon_*$ further).
Hence it is homemorphic to its image.
This implies that the set of all Suris potentials of $C^1$ distance  at most $2\delta$ from $V_S$ is compact.
Moreover, the function $\tilde{V}_S\mapsto \normC{V_S+W-\tilde{V}_S}{1}$ is continuous, and hence attains a minimum on this ball. 
Let us denote by $V^*_S$ the minimizer, and by $W^*:=V_S+W-V^*_S$.
Since the original potential $V_S$ is in the minimization domain, we know that $\normC{W^*}{1}\leq \delta$.
Now we use Proposition \ref{prop:ProjectPotential}, to find another Suris potential $\tilde{V}_S$ and a function $\tilde{W}$ such that
\[\tilde{V}_S+\tilde{W}=V^*_S+W^*=V_S+W,\]
and
\[\normC{\tilde{W}}{1}\leq C(\normC{W^*}{23})\normC{W^*}{1}^{\frac{231}{230}}.\]
Note that 
\[\normC{W^*}{23}=\normC{V_S-V^*_S+W}{23}\leq \normC{V_S-V^*_S}{23}+\normC{W}{23}\leq 2.\]
Hence, if we choose $\delta$ small enough so that 
\begin{equation}\label{eq:ChoiceOfDelta}
    C(2)\delta^{\frac{231}{230}}\leq \frac{1}{2}\delta,
\end{equation}
then we see that 
\[\normC{\tilde{W}}{1}\leq \frac{1}{2}\normC{W^*}{1}\leq\frac{1}{2}\delta.\]
Hence $\tilde{V}_S$ is inside the minimization domain, as 
\[\normC{V_S-\tilde{V}_S}{1}=\normC{\tilde{W}-W}{1}\leq \normC{\tilde{W}}{1}+\normC{W}{1}<2\delta.\]
But then the minimality of $W^*$ implies that 
\[\normC{W^*}{1}\leq\normC{\tilde{W}}{1}.\]
But together with \eqref{eq:ChoiceOfDelta} we must have $\tilde{W}=W^*=0$ which means that $V_S+W=\tilde{V}_S$ is again a  Suris potential, which finishes the proof.

\section{Proof of Corollaries \ref{cor:SpecRigid1} and \ref{cor:SpecRigid2}}\label{subsec:SpecRigid}
\label{section:rigidity_cor_proof}

Let us first prove Corollary \ref{cor:SpecRigid2} using Corollary \ref{cor:SpecRigid1}: given a $1$-parameter family of potentials $V_{\tau}$ satisfying
\[
\mathscr A(V_{\tau}) = \mathscr A(V_{0}),\qquad\tau\in I,
\]
the corresponding Mather's beta functions of each potential of the family are the same  -- see \cite[Corollary 3.2.3]{Siburg} for billiards, \cite{FKScpt} for lower regularity, both can also be applied to any twist maps. Hence
\[
\beta_{V_{\tau}} = \beta_{V_{0}},\qquad\tau\in I.
\]
Hence Corollary \ref{cor:SpecRigid2} holds if Corollary \ref{cor:SpecRigid1} is also true.\\
\vspace{0.2cm}

Let us now prove Corollary \ref{cor:SpecRigid1}. Consider $\varepsilon_*>0$, $r\geq 23$ and $\delta$ as in Theorem \ref{thm:LocalRigidity}.
Let $V_S$ be a Suris potential $V_S$ of eccentricity $\varepsilon\in(0,\varepsilon_*)$ and another potential $V$ for which $\normC{V-V_S}{r}\leq \delta$
\[
\beta_V=\beta_{V_S}.
\]

As $V_S$ is rationally integrable in $[\frac{1}{6},\frac{1}{3}]$, the map $\beta_{V_S}$ restricted to $[\frac{1}{6},\frac{1}{3}]$ is differentiable -- see \cite[Theorem 1.3.7]{Siburg}. Hence so does $\beta_{V}$. As a consequence, $F_V$ is rationally integrable in the interval $[\frac{1}{6},\frac{1}{3}]$  -- see \cite[Theorem 1.3.7]{Siburg}. By Theorem \ref{thm:LocalRigidity}, $V$ is also a Suris potential.

\section{Proof of Theorem \ref{thm:NoHigherOrderInteg}}\label{sec:rotSymmetry}
The proof is essentially a repeat of the main ideas of \cite{10.1093/imrn/rnab366}, adapted to this setting.
Suppose $y=f(x)$ is the equation of the invariant curve of rotation number $\frac{r}{k}$.
Observe that if $x_0,...,x_{k-1}$ is such a periodic orbit, then
\[x_{j+1}-2x_j+x_{j-1}=V'(x_j)\]
and by periodicity of $V$ we also get that
\[x_{j+1}+\frac{r}{k}-2(x_j+\frac{r}{k})+x_{j-1}+\frac{r}{k}=V'(x_j+\frac{r}{k}).\]
This means that $x_0+\frac{r}{k},...,x_{k-1}+\frac{r}{k}$ is also a $k$ periodic orbit. 
Hence, we can compute the corresponding vertical coordinates:
\[f(x_0)=y_0=x_1-x_0-V'(x_0)=(x_1+\frac{r}{k})-(x_0+\frac{r}{k})-V'(x_0+\frac{r}{k})=f(x_0+\frac{r}{k}).\]
Since from any point $x_0$ we have a $k$ periodic orbit, then this equality holds for all $x_0\in[0,1]$.
Hence the function $f$ is also $\frac{r}{k}$-periodic.
Denote by $R:S^1\times\R\to S^1\times \R$ the self map of the cylinder given by 
\[R(x,y)=(x+\frac{r}{k},y).\]
Then the assumption that $V$ is $\frac{r}{k}$-periodic means that $T\circ R=R\circ T$. 
Denote by $S:S^1\to S^1$ the restriction of $T$ to our invariant curve, and by $\tilde{S}:\R\to\R $ its lift, so that
\[T(x,f(x))=(S(x),f(S(x))).\]
Evaluate both sides of the equality $T\circ R= R\circ T$ on the points $(x,f(x))$:
\[T\circ R(x,f(x))=T(x+\frac{r}{k},f(x))=T(x+\frac{r}{k},f(x+\frac{r}{k}))=(S(x+\frac{r}{k}),f(S(x+\frac{r}{k}))),\]
\[R\circ T(x,f(x))=R(S(x),f(S(x)))=(\tilde{S}(x)+\frac{r}{k},f(S(x))).\]
Hence it follows that 
\[\tilde{S}(x+\frac{r}{k})=\tilde{S}(x)+\frac{r}{k}.\]
But since the orbits have rotation number $\frac{r}{k}$, we have $\tilde{S}^k(x)=x+r$.
Then these two facts imply that $\tilde{S}(x)=x+\frac{r}{k}$ (see, e.g., \cite[Lemma 2.1]{10.1093/imrn/rnab366}).
Thus the $x$ coordinates of the orbit starting at $x$ are just $\set{x+\frac{rj}{k}}_{j=0,...,{k-1}}$.
But then the Frenkel-Kontorova equation gives us:
\[V'(x)=(x+\frac{r}{k})-2x+(x-\frac{r}{k})=0,\]
so $V$ is constant.

\bibliography{NewBibliography.bib}
\bibliographystyle{abbrv}
\end{document}